\documentclass[10pt, a4]{amsproc}

\usepackage{float}
\usepackage{amscd,amsmath,amssymb,amsfonts,amsthm,ascmac}
\usepackage{enumerate,mathrsfs,stmaryrd,latexsym, comment, mathtools, mathdots} 

\usepackage[mathscr]{euscript}

\usepackage{amsthm, amssymb}
\usepackage[leqno]{amsmath}
\usepackage{caption}
\usepackage{subcaption}
\usepackage{hyperref}
\usepackage{relsize}
\usepackage{booktabs}

\usepackage[all]{xy}

\usepackage{pdflscape}

\usepackage{colonequals} %\colonequal

\usepackage{tikz-cd}
\usepackage{tikz}
\usetikzlibrary{snakes, 
	3d, matrix,decorations.pathreplacing,calc,decorations.pathmorphing,decorations.markings}
\usetikzlibrary {positioning}

\definecolor {processblue}{cmyk}{0.96,0,0,0}

\numberwithin{equation}{section}
\makeatletter
\newcommand{\leqnomode}{\tagsleft@true\let\veqno\@@leqno}
\newcommand{\reqnomode}{\tagsleft@false\let\veqno\@@eqno}
\makeatother

\allowdisplaybreaks

\newcommand{\defi}[1]{{\textit{#1}}}

\newcommand{\flag}{{\mathcal{F} \ell}}
\newcommand{\C}{{\mathbb{C}}}
\renewcommand{\P}{{\mathbb{P}}}
\newcommand{\R}{{\mathbb{R}}}

\newcommand{\Z}{{\mathbb{Z}}}

\DeclareMathOperator{\Hom}{Hom}

\DeclareMathOperator{\GL}{GL}
\DeclareMathOperator{\Conv}{Conv}

\newcommand{\Cstar}{{\C^{\ast}}}

\renewcommand{\P}{{\mathscr{P}\textup{oin}}}

\def\Xw{X_w}
\def\Yw{Y_w}

\def\C{\mathbb C}
\def\R{\mathbb R}
\def\Q{\mathbb Q}
\def\Z{\mathbb Z}

\def\ind{\mathrm{asc}}

\def\asc{\mathrm{asc}}

\def\!{\mspace{-4mu}}

\newcommand{\baru}{\bar{u}}
\newcommand{\barv}{\bar{v}}

\newcommand{\QQ}{\mathsf{Q}}

\numberwithin{equation}{section}

\theoremstyle{plain}
\newtheorem{theorem}{Theorem}[section]
\newtheorem{lemma}[theorem]{Lemma}
\newtheorem{proposition}[theorem]{Proposition}

\newtheorem{Question}[theorem]{Question}

\theoremstyle{definition}
\newtheorem{example}[theorem]{Example}

\newtheorem{remark}[theorem]{Remark}

\begin{document}
	
	\title[Poincar\'{e} polynomials of $\Yw$]{Poincar\'{e} polynomials of generic torus orbit closures in Schubert varieties}
	
	\date{\today}
	
	\author[Eunjeong Lee]{Eunjeong Lee}
	\address[E. Lee]{Center for Geometry and Physics, Institute for Basic Science (IBS), Pohang 37673, Republic of Korea}
	\email{eunjeong.lee@ibs.re.kr}
	
	\author[Mikiya Masuda]{Mikiya Masuda}
	\address[M. Masuda]{Osaka City University Advanced Mathematical Institute \&
		Department of Mathematics, Graduate School of Science, Osaka City 
		University, Sumiyoshi-ku, Sugimoto, 558-8585, Osaka, Japan}
	\email{masuda@osaka-cu.ac.jp}
	
	\author[Seonjeong Park]{Seonjeong Park}
	\address[S. Park]{Department of Mathematical Sciences, Korea Advanced Institute for Science and Technology (KAIST), 291 Daehak-ro, Yuseong-gu, Daejeon 34141, Republic of Korea}
	\email{psjeong@kaist.ac.kr}
	
	\author[Jongbaek Song]{Jongbaek Song}
	\address[J. Song]{School of Mathematics, KIAS, 85 Hoegiro Dongdaemun-gu, Seoul 02455, Republic of Korea}
	\email{jongbaek@kias.re.kr}

	%\thanks{} 
	
	\subjclass[2010]{Primary: 14M25, 14M15, secondary: 05A05}
	
	\keywords{Poincar\'{e} polynomials, projective toric varieties}

	\maketitle

	\begin{abstract}
		The closure of a generic torus orbit in the flag variety $G/B$ of type~$A$ is known to be a permutohedral variety and its Poincar\'e polynomial agrees with the Eulerian polynomial. In this paper, we study the Poincar\'e polynomial of a generic torus orbit closure in a Schubert variety in~$G/B$. When the generic torus orbit closure in a Schubert variety is smooth, its Poincar\'e polynomial is known to agree with a certain generalization of the Eulerian polynomial. We extend this result to an arbitrary generic torus orbit closure which is not necessarily smooth. 
	\end{abstract}

	%%%%%
	
	\section{Introduction}

	A toric variety of (complex) dimension $n$ is a normal algebraic variety over $\C$ containing an algebraic torus $(\Cstar)^n$ as an open dense subset, such that the action of the torus on itself extends to the whole variety. Toric varieties have a lot of symmetries, and studying toric varieties are closely related to the combinatorics on a lattice. One of the most important facts in toric geometry is that there is a bijective correspondence between projective toric varieties of dimension $n$ and full-dimensional lattice polytopes in $\R^n$. 
	
	The combinatorics of a lattice polytope $P$ encodes plentiful information on geometry and topology of the toric variety $X_P$ associated to~$P$. For instance, the face structure of $P$ captures the torus orbits of $X_P$ via the Orbit-Cone correspondence in toric varieties. 
	In particular, the smoothness of the toric variety $X_P$ can be determined by the combinatorics of the polytope~$P$; the toric variety $X_P$ is smooth if and only if the primitive edge vectors emanating from each vertex $v$ of $P$ form a $\Z$-basis for the lattice $\Z^n$. Such a polytope is called \defi{smooth}.
	
	Interestingly, when $X_P$ is smooth, both equivariant and ordinary cohomology rings of $X_P$ can be precisely computed in terms of $P$. Furthermore, the Betti numbers of $X_P$ are determined by the combinatorial type of $P$. 
	The odd-degree cohomology groups of $X_P$ vanish and the even-degree Betti numbers $b_{2k}(X_P) (= \dim H^{2k}(X_P; \Q))$ are evaluated only using the number $f_i$ of the $i$-dimensional faces of $P$ for each $i = 0,1,\dots,n$ (see~\eqref{eq_Betti_numbers}). 
	
	Unlike smooth toric varieties, the Betti numbers of a singular projective toric variety $X_P$ are not determined by the combinatorial type of~$P$ in general. Indeed, there exist two combinatorially equivalent lattice polytopes for which the associated toric varieties have different Betti numbers, see~\cite{McConnell}.  There are several approaches to computing the Betti numbers of singular toric varieties, e.g., using spectral sequences (see \cite{Fis, Jor} and \cite[\S12.3]{CLStoric}), Bia\l{}ynicki-Birula decomposition~\cite{BB}, retraction sequences on polytopes~\cite{BNSS, BSS17,  SS18}.

	In this article, we concentrate on generic torus orbit closures in Schubert varieties (in Type $A$) which are singular in general, 
	and provide an explicit formula on their Betti numbers.
	The flag variety $\flag(\C^n)$ is the set of full flags in the vector space~$\C^n$ which is identified with the homogeneous space $G/B$, where $G = \GL_n(\C)$ and $B$ is the set of upper triangular matrices in $G$. 
	The standard action of the diagonal torus~$T = (\Cstar)^n$ on the vector space~$\C^n$ induces an action of $T$ on $G/B$.
	Carrell and Kurth~\cite{CarrellKurth20} proved that every torus orbit closure in the flag variety is a normal variety. Hence all the torus orbit closures in the flag variety are toric varieties. 
	
	Schubert varieties $\Xw \coloneqq \overline{BwB/B}$ are subvarieties in the flag variety $G/B$ indexed by an element $w$ of the Weyl group $W$ ($\cong\mathfrak{S}_n$, the permutation group on $n$ elements) of $G$, and they are invariant under the action of $T$ on $G/B$. A $T$-orbit $\mathcal{O}$ in $\Xw$ is said to be \defi{generic} if its closure~$\overline{\mathcal{O}}$ and $X_w$ have the same $T$-fixed points. Such a torus orbit always exists. Although there are many generic torus orbits in~$X_w$, the isomorphism class of their closures is unique (see~\cite[Proposition~1 in \S5.2]{ge-se87}). We denote by $Y_w$ the closure of a generic torus orbit in $X_w$ and we call it the \defi{generic torus orbit closure in $\Xw$}.

	For the longest element $w_0 = n \ n-1 \ \cdots \ 1$ (in one-line notation) in~$\mathfrak{S}_n$, the generic torus orbit closure $Y_{w_0}$ is the permutohedral variety, which is a smooth projective toric variety whose fan $\Sigma$ is obtained by the Weyl chambers in type~$A_{n-1}$. Note that $\Sigma$ is the normal fan of the permutohedron, the convex hull of the points~$(u(1),\dots,u(n))$ in~$\R^n$ for all $u\in\mathfrak{S}_n$. 
	
	The notion of Bruhat interval polytope is a generalization of that of permutohedron. For two permutations $v$ and $w$ with $v\leq w$ in the Bruhat order, the Bruhat interval polytope $\QQ_{v,w}$ is the convex hull of the points $(u(1),\dots,u(n))$ in~$\R^n$ for all $v\leq u\leq w$ introduced by~\cite{ts-wi15}.
	For a permutation $w\in\mathfrak{S}_n$, each generic torus orbit closure~$\Yw$ in $\Xw$ is a projective toric variety associated with the Bruhat interval polytope
	$
	\QQ_{id, w^{-1}}.
	$
	Here, $id$ is the identity element in $\mathfrak{S}_n$. 
	Bruhat interval polytopes are not smooth in general, and the first and the second authors provide a systematic way to determine whether the Bruhat interval polytope $\QQ_{id, w^{-1}}$ is smooth or not by considering a corresponding graph (see~\cite[Theorem~1.2]{le-ma20}).
	
	The first and the second authors associate a polynomial~$A_w(t)$ to each $w\in \mathfrak{S}_n$ (see~\cite[\S8]{le-ma20} or \eqref{eq_def_of_Aw}  for the definition of $A_w(t)$) using the  poset structure on $\mathfrak{S}_n$ and show that the Poincar\'{e} polynomial of $Y_w$ agrees with $A_w(t^2)$ when $Y_w$ is smooth. The following theorem is the aim of this article, which shows that the formula holds even when $Y_w$ is singular.
	\begin{theorem}[Theorem~\ref{cor_main}]
		The toric variety $\Yw$ is paved by affine spaces, and the Poincar\'e polynomial $\P(\Yw,t)= \sum_{k} b_{k}(\Yw) t^k $ of $\Yw$ is given by 
		\[
		\P(\Yw,t) =  A_w(t^2).
		\] 
	\end{theorem}
	
	The main tool used in the proof of the foregoing theorem is \defi{retraction sequences} on polytopes (see Theorem~\ref{thm_poincare_polynomial_ascending_chains}).
	We would like to mention that one may get the same result using Bia\l{}ynicki-Birula decomposition (see~\cite{Yamanaka}).  
	
	This paper is organized as follows. In Section~\ref{backgrounds_polytopes_retraction}, we provide an overview on the relation between lattice polytopes and projective toric varieties (see Theorem~\ref{thm_fundamental_thm_of_toric_var}). Also, we present an explanation of retraction sequences,
	and how to compute the Poincar\'{e} polynomial in Theorem~\ref{thm_poincare_polynomial_ascending_chains}.
	In Section~\ref{sec_generic_torus_orbit_closures_in_Xw}, we study generic torus orbit closures in Schubert varieties and their Poincar\'e polynomials in Theorem~\ref{cor_main} whose proof is given in Section~\ref{section_proof}.
	We enclose the manuscript providing concluding remarks in Section~\ref{sec_concluding_remarks}.
	
	\smallskip
	\noindent\textbf{Acknowledgments}
	We are grateful to Hitoshi Yamanaka for explaining his approach to computing Poincar\'e polynomials of generic torus orbit closures in Schubert varieties. 
	Lee was supported by IBS-R003-D1. Masuda was supported in part by JSPS Grant-in-Aid for Scientific Research 16K05152 and a bilateral program between JSPS and RFBR. Park was supported by the Basic Science Research Program through the National Research Foundation of Korea (NRF) funded by the Government of Korea (NRF-2018R1A6A3A11047606). Song was supported by Basic Science Research Program through the National Research Foundation of Korea (NRF) funded by the Ministry of Education (NRF-2018R1D1A1B07048480) and a KIAS Individual Grant (MG076101) at Korea Institute for Advanced Study.

	%%%%
	\section{Backgrounds: polytopes and projective toric varieties} 
	\label{backgrounds_polytopes_retraction}
	The aim of this section is to study the Poincar\'{e} polynomial of certain toric varieties. The main idea is to use a  retraction sequence of a polytope introduced in \cite{BNSS, BSS17, SS18}. We first begin by reviewing some facts on projective toric varieties associated with lattice polytopes and their properties which are necessary to our aim. We refer to \cite{CLStoric} for more details. 
	
	Let $P$ be a full-dimensional polytope in $\R^n$. Then it can be defined by finitely many linear inequalities:
	\begin{equation}\label{eq_P}
	P = \{ \mathbf m \in \R^n \mid \langle \mathbf m, \mathbf v_j \rangle \leq \lambda_j \quad \text{ for } j = 1,\dots,r\}
	\end{equation}
	where $\mathbf v_j \in \R^n$ and $\lambda_j \in \R$. 
	We may assume that there are no redundant inequalities in~\eqref{eq_P}. For each vector $\mathbf v_j$ in the defining inequality, we obtain a facet, i.e.,  a codimension one face, of $P$ as follows:
	\[
	F_j = P \cap \{\mathbf m \in \R^n \mid \langle \mathbf m, \mathbf v_j \rangle = \lambda_j\}.
	\]
	We say that $P$ is a \defi{lattice polytope} if its vertices are in the lattice $\Z^n \subset \R^n$. 
	
	Given a lattice polytope $P$, one can define a toric variety $X_P$ as follows. Each face $Q$ of $P$ gives an affine toric variety ${\rm Spec}(\mathbb{C}[S_Q])$, where $S_Q$ is the semigroup generated by lattice points of $\{\mathbf{m}-\mathbf{m}' \in \mathbb{Z}^n \mid \mathbf{m} \in P\cap \mathbb{Z}^n\}$ for some lattice point $\mathbf{m}' \in Q\cap \mathbb{Z}^n$. It is straightforward to see that if $Q'$ is a face of $Q$, then ${\rm Spec}(\mathbb{C}[S_Q]) \subset {\rm Spec}(\mathbb{C}[S_{Q'}])$. Now, the toric variety $X_P$ associated to $P$ is defined by gluing these affine toric varieties with respect to the face structure of $P$. 
	
	Notice that ${\rm Spec}(\C[\Z^n])\cong (\C^\ast)^n$ is a subset of ${\rm Spec}(\mathbb{C}[S_Q])$ for all faces $Q$ of~$P$, which is identified by gluing. Hence we have ${\rm Spec}(\C[\Z^n]) \subset X_P$ which yields a natural action of $(\C^\ast)^n$ on $X_P$. 
	
	Next theorem tells us that the above association of $X_P$ from $P$ is indeed a bijection. 
	\begin{theorem}{\cite[Theorem~6.2.1]{CLStoric}}\label{thm_fundamental_thm_of_toric_var}\label{thm_bijection_P_and_X}
		There is a bijection:
		\[
		\begin{tikzcd}
		\{ P \subset \R^n \mid \text{ $P$ is  a full-dimensional lattice polytope}\} \arrow[<->,d,"1-1"]\\
		\{ (X, D) \mid  \text{$X$ is a projective toric variety, $D$ is a  $(\C^\ast)^n$-invariant ample divisor on $X$}\}
		\end{tikzcd}	
		\]
	\end{theorem}

	Because of Theorem \ref{thm_bijection_P_and_X}, various geometric and topological information of $X_P$ can be read off from the combinatorics of $P$, one of which is the smoothness of $X_P$. A projective toric variety $X_P$ is smooth if and only if the polytope $P$ is smooth (see~\cite[Theorem~2.4.3]{CLStoric}). More precisely, a vertex of an $n$-dimensional polytope $P$ is \defi{simple} if there are exactly $n$ facets intersecting the vertex $v$. When all the vertices of $P$ are simple, we call $P$ a \defi{simple polytope}. A vertex $v$ of a lattice polytope $P$ is \defi{smooth} if it is simple and the set of primitive normal vectors of facets intersecting $v$ form an integral basis of $\Z^n$; otherwise, we call $v$ \defi{singular}. A lattice polytope $P$ is called \defi{smooth} if all the vertices of $P$ are smooth. In Figure~\ref{fig:def-poly}, we present two singular polytopes and one smooth polytope. In Figure~\ref{fig:def-poly_2}, consider the following vertex
	\[
	\{(0,2)\} = \{ x \in \R^2 \mid  \langle x, (-1,0) \rangle = 0\} \cap \{x \in \R^2 \mid \langle x, (1,2) \rangle = 4 \}.
	\]
	Since the set $\{(-1,0), (1,2)\}$ does not form a $\Z$-basis of $\Z^2$, the vertex $(0,2)$ is singular. 
	\begin{figure}[b]
		\begin{subfigure}[b]{.4\textwidth}
			\begin{center}
				\begin{tikzpicture}[scale=0.4]
				\filldraw[draw=black,fill=lightgray] (0,0)--(2,-1)--(3,0.5)--(1.5,2)--cycle;
				\draw[dotted](0,0)--(1.2,0.8)--(3,0.5);
				\draw[dotted] (1.5,2)--(1.2,0.8);
				\draw (2,-1)--(1.5,2);
				\end{tikzpicture}
			\end{center}
			\caption{Not simple (so singular).}
		\end{subfigure}
		\begin{subfigure}[b]{.3\textwidth}
			\begin{center}
				\begin{tikzpicture}[scale=0.4]
				\draw[gray, ->] (-1,0) -- (5,0);
				\draw[gray, ->] (0,-1)--(0,3);
				\filldraw[draw=black,fill=lightgray] (0,0)--(4,0)--(0,2)--cycle;
				\foreach \x in {-1,0,...,5}
				\foreach \y in {-1,0,...,3}
				{
					\fill (\x,\y) circle (2pt);
				}

				\end{tikzpicture}
			\end{center}
			\caption{Simple but singular.}\label{fig:def-poly_2}
		\end{subfigure}
		\begin{subfigure}[b]{.25\textwidth}
			\begin{center}
				\begin{tikzpicture}[scale=0.4]
				\draw [gray, ->] (-1,0)--(3,0);
				\draw [gray, ->] (0,-1) --(0,3);
				\filldraw[draw=black,fill=yellow] (0,0)--(2,0)--(0,2)--cycle;
				\foreach \x in {-1,0,...,3}
				\foreach \y in {-1,0,...,3}
				{
					\fill (\x,\y) circle (2pt);
				}
				
				\end{tikzpicture}
			\end{center}
			\caption{Smooth.}
		\end{subfigure}
		\caption{Example and non-examples of smooth lattice polytopes.}
		\label{fig:def-poly}
	\end{figure}
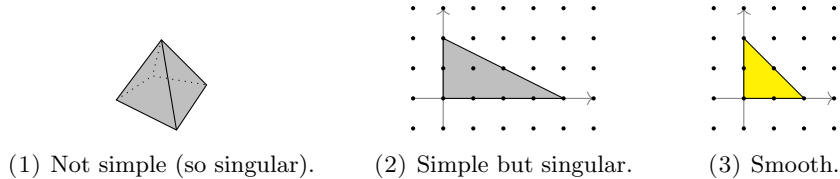
	
	For the case of a smooth toric variety $X_P$, it is well-known that the ordinary cohomology of $X_P$ is concentrated in even degrees and  
	the Betti numbers
	$
	b_{2k}(X_P)  = \dim H^{2k}(X_P,\Q)
	$
	are evaluated as follows (see~\cite[Theorem~12.3.12]{CLStoric})
	\begin{equation}\label{eq_Betti_numbers}
	b_{2k}(X_P) = \sum_{i=k}^n(-1)^{i-k} {i \choose k} f_i,
	\end{equation}
	where $f_i$ is the number of the $i$-dimensional faces of $P$.
	We refer the readers to \cite[Chapter 5]{BP15toric} and \cite[Chapter 12]{CLStoric} for more details. 
	
	Turning our attention to singular toric varieties, the relation \eqref{eq_Betti_numbers} is no longer true in general. Indeed, as we mentioned in Introduction, there are two toric varieties with different Betti numbers, but they are associated with combinatorially equivalent polytopes, see \cite{McConnell}. However, if a polytope $P$ has a \emph{retraction sequence} which we shall define below, then the Betti numbers can be captured from the combinatorics of $P$.

	Consider a sequence of triples $\{(P_i, Q_i, v_i)\}$ defined inductively whose initial term $(P_1, Q_1, v_1)=(P, P, v)$ for some simple vertex $v$ of $P$. For $i \geq 2$, $P_i$ is the union of all faces of $P_{i-1}$ not containing $v_{i-1}$ and we choose a vertex $v_i$ in $P_i$ and the maximal dimensional face $Q_i$ of $P_i$ such that $v_i$ is simple in $Q_i$. We call~$\{(P_i, Q_i, v_i)\}_{i=1}^\ell$ a \emph{retraction sequence of $P$} if it ends up with $(v_\ell, v_\ell, v_\ell)$ for some vertex $v_\ell$ of $P$ and $\ell$ is the number of vertices of $P$. 
	
	\begin{example}\label{ex_pyramid}
		Let $P$ be a pyramid with five vertices:
		\[
		P = \Conv\{ (1,0,0), (0,1,0), (-1,0,0), (0,0,-1), (0,0,1) \}.
		\] 
		A retraction sequence of $P$ is illustrated in the figure below, where the simple vertex $v_i$ and the maximal face  $Q_i$ containing $v_i$  for each term are highlighted. 
		
		\tikzset{  mid arrow/.style={postaction={decorate,decoration={
						markings,
						mark=at position .5 with {\arrow[#1]{stealth}}
		}}}}
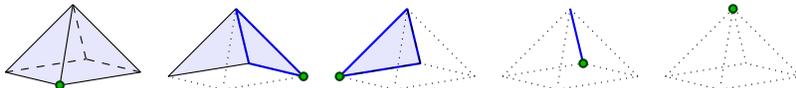
		\begin{table}[H]
			\begin{tabular}{ccccc}
				\begin{tikzpicture}[line join=bevel,z=-5.5, scale = 0.9]
				\coordinate (A1) at (0,0,-1);
				\coordinate (A2) at (-1,0,0);
				\coordinate (A3) at (0,0,1);
				\coordinate (A4) at (1,0,0);
				\coordinate (B1) at (0,1,0);
				\coordinate (C1) at (0,-1,0);

				\fill[fill=blue!95!black, fill opacity=0.100000] (A2)--(A3)--(A4)--(B1)--(A2);

				\draw[dashed] (A2)--(A1);
				\draw[dashed] (A4)--(A1);
				\draw[dashed] (A1)--(B1);
				
				\draw  (A3)--(A4);
				\draw  (A3)--(A2);
				\draw  (A3)--(B1);
				\draw  (A2)--(B1);
				\draw  (A4)--(B1);
				
				\node[inner sep=1pt,circle,draw=green!25!black,fill=green!75!black,thick] at (A3) {};
				
				\end{tikzpicture}
				& 	
				\begin{tikzpicture}[line join=bevel,z=-5.5, scale = 0.9,
				edge/.style={color=blue!95!black, thick},
				facet/.style={fill=blue!95!black,fill opacity=0.100000}]
				\coordinate (A1) at (0,0,-1);
				\coordinate (A2) at (-1,0,0);
				\coordinate (A3) at (0,0,1);
				\coordinate (A4) at (1,0,0);
				\coordinate (B1) at (0,1,0);
				\coordinate (C1) at (0,-1,0);

				\filldraw [fill=blue!95!black, fill opacity=0.100000, draw = blue!95!black, thick] (A1)--(A4)--(B1)--cycle;

				\draw[fill=blue!95!black, fill opacity=0.100000] (B1)--(A2)--(A1);
				\draw[dotted] (A4)--(A1);
				\draw[dotted] (A1)--(B1);
				
				\draw[dotted]  (A3)--(A4);
				\draw[dotted]  (A3)--(A2);
				\draw[dotted]  (A3)--(B1);
				\draw[dotted]  (A4)--(B1);
				
				\node[inner sep=1pt,circle,draw=green!25!black,fill=green!75!black,thick] at (A4) {};
				
				\end{tikzpicture}
				&
				\begin{tikzpicture}[line join=bevel,z=-5.5, scale = 0.9]
				\coordinate (A1) at (0,0,-1);
				\coordinate (A2) at (-1,0,0);
				\coordinate (A3) at (0,0,1);
				\coordinate (A4) at (1,0,0);
				\coordinate (B1) at (0,1,0);
				\coordinate (C1) at (0,-1,0);

				\filldraw [fill=blue!95!black, fill opacity=0.100000, draw = blue!95!black, thick](A1)--(A2)--(B1)--cycle;

				\draw[dashed] (A2)--(A1);
				\draw[dashed] (A1)--(B1);
				
				\draw[dotted] (A4)--(A1);
				\draw[dotted]  (A3)--(A4);
				\draw[dotted]  (A3)--(A2);
				\draw[dotted]  (A3)--(B1);
				\draw[dotted]  (A4)--(B1);
				
				\node[inner sep=1pt,circle,draw=green!25!black,fill=green!75!black,thick] at (A2) {};
				
				\end{tikzpicture}
				&
				\begin{tikzpicture}[line join=bevel,z=-5.5, scale = 0.9]
				\coordinate (A1) at (0,0,-1);
				\coordinate (A2) at (-1,0,0);
				\coordinate (A3) at (0,0,1);
				\coordinate (A4) at (1,0,0);
				\coordinate (B1) at (0,1,0);
				\coordinate (C1) at (0,-1,0);

				\draw[fill=blue!95!black, fill opacity=0.100000, draw = blue!95!black, thick] (A1)--(B1);
				
				\draw[dotted] (A2)--(B1);
				\draw[dotted] (A2)--(A1);		
				\draw[dotted] (A4)--(A1);
				\draw[dotted]  (A3)--(A4);
				\draw[dotted]  (A3)--(A2);
				\draw[dotted]  (A3)--(B1);
				\draw[dotted]  (A4)--(B1);		
				\node[inner sep=1pt,circle,draw=green!25!black,fill=green!75!black,thick] at (A1) {};
				
				\end{tikzpicture}
				&	\begin{tikzpicture}[line join=bevel,z=-5.5, scale = 0.9]
				\coordinate (A1) at (0,0,-1);
				\coordinate (A2) at (-1,0,0);
				\coordinate (A3) at (0,0,1);
				\coordinate (A4) at (1,0,0);
				\coordinate (B1) at (0,1,0);
				\coordinate (C1) at (0,-1,0);

				\draw[dotted] (A1)--(B1);
				\draw[dotted] (A2)--(B1);
				\draw[dotted] (A2)--(A1);		
				\draw[dotted] (A4)--(A1);
				\draw[dotted]  (A3)--(A4);
				\draw[dotted]  (A3)--(A2);
				\draw[dotted]  (A3)--(B1);
				\draw[dotted]  (A4)--(B1);				
				\node[inner sep=1pt,circle,draw=green!25!black,fill=green!75!black,thick] at (B1) {};
				
				\end{tikzpicture}
			\end{tabular}
		\end{table}

	\end{example}
	
	\begin{remark}
		Every simple polytope has at least one retraction sequence (see~\cite[Proposition~2.3]{BSS17}).
		However, some singular polytopes, for instance, the rhombododecahedron appeared in \cite{McConnell} do not admit a retraction sequence. (Also, the polytope in Figure~\ref{figure_BIP_1324_4231} does not admit a retraction sequence.) Hence it would be worthwhile to study necessary and sufficient conditions for a polytope to have a retraction sequence (see~Question~\ref{question_1}).
	\end{remark}
	
	One of the information encoded in a retraction sequence of a lattice polytope $P$ is the Poincar\'{e} polynomial of $X_P$. 
	
	\begin{theorem}[{\textup see~\cite[Proposition B.3]{le-ma20} and reference therein}]\label{thm_poincare_polynomial_ascending_chains}
		Let $P$ and $X_P$ be as above. If $P$ admits a retraction sequence $\{(P_i, Q_i, v_i)\}_{i=1}^\ell$, then 
		the Poincar\'e polynomial $\P(X_P,t)$ of $X_P$ is given by 
		\begin{equation}\label{eq_poin_poly_ret}
		\P(X_P,t)=\sum_{i=1}^\ell t^{2\dim Q_i}.
		\end{equation}
	\end{theorem}

	We now discuss a specific way to obtain a retraction sequence for a lattice polytope $P$, which enables us to describe \eqref{eq_poin_poly_ret} in a more concrete formula. 
	Let $h\colon \R^n\to \R$ be a linear function satisfying that $h(u) \neq h(v)$ if two vertices $u$ and~$v$ are joined by an edge in $P$.  Then the function $h$ gives an orientation on each edge of~$P$, namely, we give an orientation of the edge connecting two vertices $u$ and $v$ of~$P$ by $u \to v$ if $h(u) < h(v)$. 
	For each vertex $u$ of~$P$, we define 
	\begin{equation*}
	\ind(u)\coloneqq \# \{ u \to v \}.
	\end{equation*}

	\begin{lemma}\label{lem_height_ftn_defines_ret_seq}
		Let $P$ and $h$ be as above.
		For each vertex $u$ of $P$, if the direction vectors of ascending edges emanating from $u$ are linearly independent and form a face of $P$, 
		then there exists a retraction sequence of $P$ induced from $h$. 
	\end{lemma}
	\begin{proof}
		Let $Q(u)$ be the face defined by the edges emanating from $u$. Then, the linearly independence of the direction vectors implies that 
		$\dim Q(u) =\ind(u)$.
		Moreover, the subset $\hat{Q}(u)$ obtained by removing from $Q(u)$ all faces not containing~$u$ is homeomorphic to~$\R_\geq^{\ind(u)}$ and 
		\begin{equation*}\label{eq_polytope_decomp_affine_pieces}
		P=\bigsqcup_{u\in V(P)}\hat{Q}(u).
		\end{equation*}		
		
		We now choose a total order $u_1,\dots,u_{\ell}$ on the vertices of $P$ such that
		\[
		h(u_1) \le h(u_2) \le \cdots \le h(u_\ell)
		\]
and  define a retraction sequence by initiating $(P_1, Q_1, v_1)=(P, P, u_1)$.  For $i\geq 2$, we set 
		\begin{equation*}\label{eq_def_ret_seq_from_h}
		(P_i, Q_i, v_i)\colonequals \left( P_{i-1} \setminus \hat{Q}(u_{i-1}), Q(u_i), u_i\right).
		\end{equation*}
		Then the discussion above establishes the claim.
	\end{proof}
	
	\begin{remark}\label{rmk_existence_of_F}
		The hypothesis about the existence of the face of $P$ containing the edges emanating from the vertex $u$ in  Lemma \ref{lem_height_ftn_defines_ret_seq} can be satisfied if one finds a function $F \colon P\subset \R^n \to \R$ such that 
		\begin{equation}\label{eq_con_of_F}
		\begin{cases}
		F(u - v) = 0 &\text{ if } u \to v, \\
		F(u- v) > 0 &\text{ if } v \to u.
		\end{cases}
		\end{equation}
		Indeed, the desired face is given by
		\[
		Q(u)\colonequals P \cap (u + \ker F).
		\] 
		We note that the existence of such a function $F$ will be specified for the polytopes in which we are interested for the main result of this paper. 
	\end{remark}
	
	The following theorem is straightforward from Theorem \ref{thm_poincare_polynomial_ascending_chains} and Lemma \ref{lem_height_ftn_defines_ret_seq}. 
	\begin{theorem}\label{thm_poincare_polynomial_height_ftn}
		If there exists a  linear function $h$ satisfying the hypothesis of Lemma \ref{lem_height_ftn_defines_ret_seq},  then the Poincar\'e polynomial $\P(X_P,t)$ of $X_P$ is given by 
		\[
		\P(X_P,t) =\sum_{u\in V(P)}t^{2\cdot \ind(u)}.
		\] 
	\end{theorem}
	
	\begin{example}
		We continue to consider the polytope in Example \ref{ex_pyramid}. 
		Choose a linear function $h \colon \R^3 \to \R$ defined by 
		\[
		h(x_1,x_2,x_3) = -2x_1 - x_2 + 3 x_3.
		\] 
		The function $h$ gives an orientation on edges of $P$ as displayed in Figure~\ref{fig_pyramid}.
		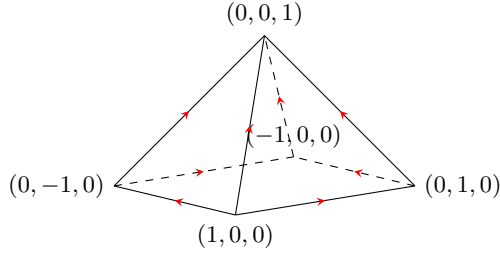
\begin{figure}
			\tikzset{  mid arrow/.style={postaction={decorate,decoration={
							markings,
							mark=at position .5 with {\arrow[#1]{stealth}}
			}}}}
			\begin{tikzpicture}[line join=bevel,z=-5.5, scale = 2]
			\coordinate (A1) at (0,0,-1);
			\coordinate (A2) at (-1,0,0);
			\coordinate (A3) at (0,0,1);
			\coordinate (A4) at (1,0,0);
			\coordinate (B1) at (0,1,0);
			\coordinate (C1) at (0,-1,0);

			\node[above] at (A1) {\small $(-1,0,0)$};
			\node[left] at (A2) {\small $(0,-1,0)$};
			\node[below] at (A3) {\small $(1,0,0)$};
			\node[right] at (A4) {\small $(0,1,0)$};
			\node[above]  at (B1) {\small$(0,0,1)$};

			\draw[dashed, mid arrow=red] (A2)--(A1);
			\draw[dashed, mid arrow=red] (A4)--(A1);
			\draw[dashed, mid arrow=red] (A1)--(B1);
			
			\draw [mid arrow=red] (A3)--(A4);
			\draw [mid arrow=red] (A3)--(A2);
			\draw [mid arrow = red] (A3)--(B1);
			\draw [mid arrow = red] (A2)--(B1);
			\draw [mid arrow = red] (A4)--(B1);
			
			\end{tikzpicture}
			\caption{The pyramid with an orientation.}
			\label{fig_pyramid}
		\end{figure}
		Then, for each vertex $u$ of $P$, the corresponding face $Q(u)$ and the number of ascending edges are given in Table~\ref{table_retraction_pyramid}.
		\tikzset{  mid arrow/.style={postaction={decorate,decoration={
						markings,
						mark=at position .5 with {\arrow[#1]{stealth}}
		}}}}
		\begin{table}[H]
			\renewcommand{\arraystretch}{1.2}
			\begin{tabular}{|c|c|c|c|c|c|}
				\hline
				\raisebox{1em}{$Q(u)$} & 	\begin{tikzpicture}[line join=bevel,z=-5.5, scale = 0.9]
				\coordinate (A1) at (0,0,-1);
				\coordinate (A2) at (-1,0,0);
				\coordinate (A3) at (0,0,1);
				\coordinate (A4) at (1,0,0);
				\coordinate (B1) at (0,1,0);
				\coordinate (C1) at (0,-1,0);

				\fill[fill=blue!95!black, fill opacity=0.100000] (A2)--(A3)--(A4)--(B1)--(A2);

				\draw[dashed, mid arrow=red] (A2)--(A1);
				\draw[dashed, mid arrow=red] (A4)--(A1);
				\draw[dashed, mid arrow=red] (A1)--(B1);
				
				\draw [mid arrow=red] (A3)--(A4);
				\draw [mid arrow=red] (A3)--(A2);
				\draw [mid arrow = red] (A3)--(B1);
				\draw [mid arrow = red] (A2)--(B1);
				\draw [mid arrow = red] (A4)--(B1);
				
				\node[inner sep=1pt,circle,draw=green!25!black,fill=green!75!black,thick] at (A3) {};
				
				\end{tikzpicture}
				& 	\begin{tikzpicture}[line join=bevel,z=-5.5, scale = 0.9,
				edge/.style={color=blue!95!black, thick},
				facet/.style={fill=blue!95!black,fill opacity=0.100000}]
				\coordinate (A1) at (0,0,-1);
				\coordinate (A2) at (-1,0,0);
				\coordinate (A3) at (0,0,1);
				\coordinate (A4) at (1,0,0);
				\coordinate (B1) at (0,1,0);
				\coordinate (C1) at (0,-1,0);

				\filldraw [fill=blue!95!black, fill opacity=0.100000, draw = blue!95!black, thick] (A1)--(A4)--(B1)--cycle;

				\draw[dashed, mid arrow=red] (A2)--(A1);
				\draw[dashed, mid arrow=red] (A4)--(A1);
				\draw[dashed, mid arrow=red] (A1)--(B1);
				
				\draw [mid arrow=red] (A3)--(A4);
				\draw [mid arrow=red] (A3)--(A2);
				\draw [mid arrow = red] (A3)--(B1);
				\draw [mid arrow = red] (A2)--(B1);
				\draw [mid arrow = red] (A4)--(B1);
				
				\node[inner sep=1pt,circle,draw=green!25!black,fill=green!75!black,thick] at (A4) {};
				
				\end{tikzpicture}
				&
				\begin{tikzpicture}[line join=bevel,z=-5.5, scale = 0.9]
				\coordinate (A1) at (0,0,-1);
				\coordinate (A2) at (-1,0,0);
				\coordinate (A3) at (0,0,1);
				\coordinate (A4) at (1,0,0);
				\coordinate (B1) at (0,1,0);
				\coordinate (C1) at (0,-1,0);

				\filldraw [fill=blue!95!black, fill opacity=0.100000, draw = blue!95!black, thick](A1)--(A2)--(B1)--cycle;

				\draw[dashed, mid arrow=red] (A2)--(A1);
				\draw[dashed, mid arrow=red] (A4)--(A1);
				\draw[dashed, mid arrow=red] (A1)--(B1);
				
				\draw [mid arrow=red] (A3)--(A4);
				\draw [mid arrow=red] (A3)--(A2);
				\draw [mid arrow = red] (A3)--(B1);
				\draw [mid arrow = red] (A2)--(B1);
				\draw [mid arrow = red] (A4)--(B1);
				
				\node[inner sep=1pt,circle,draw=green!25!black,fill=green!75!black,thick] at (A2) {};
				
				\end{tikzpicture}
				&
				\begin{tikzpicture}[line join=bevel,z=-5.5, scale = 0.9]
				\coordinate (A1) at (0,0,-1);
				\coordinate (A2) at (-1,0,0);
				\coordinate (A3) at (0,0,1);
				\coordinate (A4) at (1,0,0);
				\coordinate (B1) at (0,1,0);
				\coordinate (C1) at (0,-1,0);

				\draw[fill=blue!95!black, fill opacity=0.100000, draw = blue!95!black, thick] (A1)--(B1);

				\draw[dashed, mid arrow=red] (A2)--(A1);
				\draw[dashed, mid arrow=red] (A4)--(A1);
				\draw[dashed, mid arrow=red] (A1)--(B1);
				
				\draw [mid arrow=red] (A3)--(A4);
				\draw [mid arrow=red] (A3)--(A2);
				\draw [mid arrow = red] (A3)--(B1);
				\draw [mid arrow = red] (A2)--(B1);
				\draw [mid arrow = red] (A4)--(B1);
				
				\node[inner sep=1pt,circle,draw=green!25!black,fill=green!75!black,thick] at (A1) {};
				
				\end{tikzpicture}
				&	\begin{tikzpicture}[line join=bevel,z=-5.5, scale = 0.9]
				\coordinate (A1) at (0,0,-1);
				\coordinate (A2) at (-1,0,0);
				\coordinate (A3) at (0,0,1);
				\coordinate (A4) at (1,0,0);
				\coordinate (B1) at (0,1,0);
				\coordinate (C1) at (0,-1,0);

				\draw[dashed, mid arrow=red] (A2)--(A1);
				\draw[dashed, mid arrow=red] (A4)--(A1);
				\draw[dashed, mid arrow=red] (A1)--(B1);
				
				\draw [mid arrow=red] (A3)--(A4);
				\draw [mid arrow=red] (A3)--(A2);
				\draw [mid arrow = red] (A3)--(B1);
				\draw [mid arrow = red] (A2)--(B1);
				\draw [mid arrow = red] (A4)--(B1);
				
				\node[inner sep=1pt,circle,draw=green!25!black,fill=green!75!black,thick] at (B1) {};
				
				\end{tikzpicture}
				\\
				\hline
				$h(u)$ & -2 & -1 & 1 & 2 & 3  \\ 
				\hline
				$\asc(u)$ & 3 & 2 & 2 & 1 & 0\\
				\hline
			\end{tabular}
			\caption{A retraction sequence of the pyramid.}
			\label{table_retraction_pyramid}
		\end{table}
		\noindent Finally, the Poincar\'e polynomial $\P(X_P,t)$ of $X_P$ is given by
		\[
		\P(X_P, t) = 1  + t^2 + 2 t^4 + t^6. 
		\]
	\end{example}
	
	We finish this section by describing a geometric interpretation of Theorem \ref{thm_poincare_polynomial_height_ftn}. 
	An algebraic variety $X$ is called \emph{paved by (complex) affine spaces} if $X$ has a filtration 
	\[
	X=X_\ell \supset X_{\ell-1} \supset \cdots \supset X_1\supset X_0=\emptyset
	\]
	by closed subvarieties such that $X_i \setminus X_{i-1}$ is isomorphic to a disjoint union of affine spaces. If an algebraic variety	is paved by affine spaces, then the cycle map 
	\[
	cl_X \colon A_\ast(X) \to H_\ast(X;\Z)
	\] 
	is an isomorphism and the Betti numbers are directly computed by counting dimensions of affine spaces. We refer to \cite[Example 19.1.11]{Ful-Inter}.

	\begin{proposition}\label{prop_affine_paving}
		Let $P$ be a full-dimensional lattice polytope in $\mathbb{R}^n$. If a lattice polytope $P$ admits a retraction sequence $\{(P_i, Q_i, v_i)\}_{i=1}^\ell$ such that $v_i$ is smooth in~$Q_i$, then the associated toric variety $X_P$ is paved by affine spaces. 
	\end{proposition}
	\begin{proof}
		Each face $Q$ of $P$ defines a sublattice $\Z_Q$ of $\mathbb{Z}^n$ generated by the lattice points in $Q\cap \mathbb{Z}^n$, which yields a toric subvariety $X_Q$ of $X_P$. On the other hand, there is a bijection between the set of faces of $P$ and the set of $(\mathbb{C}^\ast)^n$-orbits of $X_P$, see~\cite[Theorem 3.2.6]{CLStoric}. Let $\mathcal{O}(\sigma_Q)$ be the $(\mathbb{C}^\ast)^n$-orbit corresponding to a face $Q$ of $P$. To be more precise, $\mathcal{O}(\sigma_Q)$ is isomorphic to the torus $\Hom_\Z(\Z_Q, \C^\ast)\cong (\C^\ast)^{\dim Q}$. These two objects $X_Q$ and $\mathcal{O}(\sigma_Q)$ are related by \cite[Proposition 3.2.9]{CLStoric}, which asserts that $X_Q$ is isomorphic to $\bigcup_{Q' \text{: face of } Q}\mathcal{O}(\sigma_{Q'})$. 
		
		Now, we assume that there is a retraction sequence $\{(P_i, Q_i, v_i)\}_{i=1}^\ell$ satisfying the hypothesis. We set 
		\[
		X_i:=\bigcup_{Q\underset{\text{face}}{\subset}P_{\ell-i+1}}X_Q \cong \bigcup_{Q\underset{\text{face}}{\subset}P_{\ell-i+1}}\mathcal{O}(\sigma_Q)
		\]
		for $i=1, \dots, \ell$. 
		For example, $X_\ell = X_P$ and $X_{\ell-1}$ is isomorphic to the union of all orbits $\mathcal{O}({\sigma_Q})$ such that $Q$ is a face of $P$ not containing $v_1$. Therefore, 
		\begin{equation}\label{eq_affine_paving}
		X_\ell \setminus X_{\ell-1}\cong  \mathcal{O}({\sigma_Q})\cong {\rm Spec}(\C[S_{v_1}]) \cong \mathbb{A}^n.
		\end{equation}
		The second and the third isomorphisms follow from \cite[Theorem 3.2.6-(c)]{CLStoric} and the smoothness of $v$, respectively. In general, we claim that $X_i  \setminus X_{i-1}\cong \mathbb{A}^{\dim Q_{\ell-i+1}}$ for each $i\geq 1$, which is proved by a similar identification to \eqref{eq_affine_paving}. 
	\end{proof}

	\section{Generic torus orbit closures in Schubert varieties and \\
		their Poincar\'e polynomials}
	\label{sec_generic_torus_orbit_closures_in_Xw}
	
	Let $G = \GL_n(\C)$ and $B$ the set of upper triangular matrices in $G$. 
	Let $T$ be the set of diagonal matrices in $G$, accordingly $T \cong (\Cstar)^n$. 
	The flag variety~$\flag(\C^n) = G/B$ admits an action of the diagonal torus $T$ via the left multiplication, and $(G/B)^T = \{ uB \mid u \in W\}$, where $W$ is the Weyl group of $G$ which is identified with the group $\mathfrak{S}_n$ of permutations of $n$ elements.
	Note that the subtorus $\{\textup{diag}(t,t,\dots,t) \mid t \in \Cstar\} \subset T$ acts trivially on $\flag(\C^n)$.
	For each~$w \in \mathfrak{S}_n$, we denote by $\Xw$ the Schubert variety $\overline{BwB/B} \subset G/B$. 
	We consider the Kirillov--Kostant--Souriau symplectic form $\omega$ on $G/B$ determined by the sum of all fundamental weights. Then the corresponding moment map $\mu \colon \flag(\C^n) \to \R^n$ has the following value at each fixed point $uB \in G/B$:
	\[
	\mu(uB) = (u^{-1}(1),\dots, u^{-1}(n)) \in \R^n 
	\]
	for $u \in \mathfrak{S}_n$.
	See~\cite[Section~3]{LMP1} and reference therein for more details.
	For notational simplicity, we denote by $\baru \in \R^n$ the moment map image $\mu(uB)$ for each $u \in \mathfrak{S}_n$. 
	Since $u$ is a permutation, the image $\baru$ lies in the hyperplane 
	\begin{equation}\label{eq_hyperplane}
	\left\{(x_1,\dots,x_n) \in \R^n \,\middle|\, \sum_{i=1}^n x_i = \frac{n(n+1)}{2} \right\}.
	\end{equation} 
	Hence the whole image $\mu(G/B)$ is contained in the hyperplane~\eqref{eq_hyperplane}. 
	
	Let $\Yw$ be the generic torus orbit closure in $\Xw$. Then 
	\[
	\Yw^T = \Xw^T  = \{ uB \mid id \leq u \leq w\},
	\]
	and the image of $\Yw$ under the moment map $\mu$ is a Bruhat interval polytope
	\[
	\mu(\Yw) = \Conv\{ \baru \mid id \leq u \leq w\} = \QQ_{id, w^{-1}}. 
	\]
	Here, the second equality follows from the fact that $v \leq w $ if and only if $v^{-1} \leq w^{-1}$ (cf. \cite[Corollary~2.2.5]{BjornerBrenti}). 
	Furthermore, the set of vertices of $\QQ_{id, w^{-1}}$ is the same as the image of the $T$-fixed points
	\[
	\text{vertices of $\QQ_{id, w^{-1}}$} = \{\baru \mid id \leq u \leq w\},
	\]
	and the toric variety $\Yw$ is associated to the polytope $\QQ_{id, w^{-1}}$ via the bijective correspondence in Theorem~\ref{thm_fundamental_thm_of_toric_var}.
	Note that for $v,w\in\mathfrak{S}_n$, the Bruhat interval polytope $\QQ_{v,w}$ is defined whenever~$v\leq w$ in the Bruhat order.
	\begin{lemma}\label{lemm:smooth-retraction}
		If $\{(P_i,Q_i,v_i)\}_{i=1}^\ell$ is a retraction sequence on a Bruhat interval polytope~$\QQ_{v,w}$, then each vertex $v_i$ is smooth in $Q_i$.
	\end{lemma}
	\begin{proof}
		Let $\{(P_i,Q_i,v_i)\}_{i=1}^\ell$ be a retraction sequence of~$\QQ_{v,w}$. Then $Q_i$'s are Bruhat interval polytopes for all $i=1,\dots,\ell$ by~\cite[Theorem~4.1]{ts-wi15}. From the definition of a retraction sequence, each vertex $v_i$ is simple in the polytope~$Q_i$. Then, the vertex $v_i$ is smooth in $Q_i$ by~\cite[Proposition~4.6]{LMP1}. 
	\end{proof}
Combining Proposition~\ref{prop_affine_paving} and Lemma~\ref{lemm:smooth-retraction}, we obtain that if there is a retraction sequence on the Bruhat interval polytope~$\QQ_{id,w^{-1}}$, then the generic torus orbit closure $\Yw$ in $\Xw$ is paved by affine spaces.

	In the remaining part of this section, we shall observe that the Bruhat interval polytope $\QQ_{id,w^{-1}}$ has a retraction sequence, and compute the Poincar\'e polynomial of $\Yw$ in Theorem~\ref{cor_main}.
	To see the face structure of Bruhat interval polytopes, we recall the description of edges of $\QQ_{id, w^{-1}}$ from~\cite{le-ma20}.
	For $u\le w$, we set
	\[
	\widetilde{E}_w(u) =\{ (u(i),u(j))\mid i<j,\ t_{u(i),u(j)}u\le w,\ |\ell(u)-\ell(t_{u(i),u(j)}u)|=1\},
	\]
	where $t_{a,b}$ denotes the transposition interchanging $a$ and $b$. Moreover, $\ell(v)$ denotes the length of a permutation $v$. 
	We associate a directed graph~$G^{e,w}_u$ with $\widetilde{E}_w(u)$ such that the vertices are $[n] = \{1,2,\dots,n\}$ and the directed edges are given by elements in $\widetilde{E}_w(u)$. Then this graph $G^{e,w}_u$ is an acyclic digraph by~\cite[Theorem~4.19]{ts-wi15}, and we denote  by $E_w(u)$ the edge set of its transitive reduction. 
	Here, a \defi{transitive reduction} of a digraph is another digraph with the same vertices and as fewer edges as possible, such that if there is a directed path connecting two vertices, then there is also such a path in the reduction. Note that a transitive reduction of a finite acyclic digraph is unique (see~\cite{a-ga-ul72}). Since our graph $G^{e,w}_u$ is an acyclic digraph, the set $E_w(u)$ is well-defined by the uniqueness of transitive reduction of a finite acyclic digraph. 
	We call elements in $E_w(u)$ \defi{indecomposable}, and elements in $\widetilde{E}_w(u)\setminus E_w(u)$ \defi{decomposable}.
	\begin{example}
		Take $w = 3412$ and $u = 2143$. Then we have that 
		\[
		\widetilde{E}_w(u) = \{(1,4), (2,3), (2,1), (4,3)\}.
		\] 
		The corresponding graph $G^{e,w}_u$ and its transitive reduction are presented as follows.
		
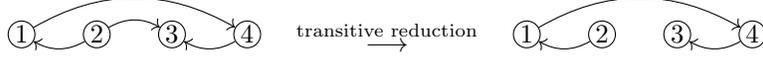
\begin{center}
			\begin{tikzpicture}[node/.style={circle,draw, fill=white!20, inner sep = 0.25mm}]
			\node[node] (1) at (1,1) {$1$};
			\node[node] (2) at (2,1) {$2$};
			\node[node] (3) at (3,1) {$3$};
			\node[node] (4) at (4,1) {$4$};
			
			\draw[->] (1) to [bend left] (4);
			\draw[->] (4) to [bend left] (3);
			\draw[->] (2) to [bend left] (1);
			\draw[->] (2) to [bend left] (3);
			\end{tikzpicture}
			$\quad \stackrel{\text{transitive reduction}}{\longrightarrow} \quad $
			\begin{tikzpicture}[node/.style={circle,draw, fill=white!20, inner sep = 0.25mm}]
			\node[node] (1) at (1,1) {$1$};
			\node[node] (2) at (2,1) {$2$};
			\node[node] (3) at (3,1) {$3$};
			\node[node] (4) at (4,1) {$4$};
			
			\draw[->] (1) to [bend left] (4);
			\draw[->] (4) to [bend left] (3);
			\draw[->] (2) to [bend left] (1);
			
			\end{tikzpicture}
		\end{center}
		As a result, $E_w(u) = \{(1,4), (2,1), (4,3)\}$, and the element $(2,3)$ is decomposable. 
	\end{example}
	
	The elements in $E_w(u)$ correspond to the edges of $\QQ_{id,w^{-1}}$ emanating from the vertex $\baru$. Indeed, the following is observed in~\cite[Proposition 7.7]{le-ma20}.
	\begin{proposition}
		In the polytope $\QQ_{id, w^{-1}}$, two vertices $\baru$ and $\barv$ are joined by an edge if and only if 
		\[
		v=t_{u(i),u(j)}u(=ut_{i,j}) \quad \text{ for }(u(i),u(j))\in E_w(u).
		\] 
	\end{proposition}
	
	Now we take a linear function $h$ appropriately such that it satisfies the conditions in Lemma~\ref{lem_height_ftn_defines_ret_seq}.
	
	\begin{lemma} \label{lemm:2-1}
		Let $h\colon \R^n\to\R$ be a linear function defined by the inner product with a vector $(a_1,\dots,a_n) \in \R^n$ with $a_1>a_2>\dots>a_n$. Then for $(u(i),u(j))\in E_w(u)$, the edge connecting $\baru$ and $\barv$, where $v=t_{u(i),u(j)}u$, is ascending with respect to the function $h$ if and only if $u(i)<u(j)$.
	\end{lemma}
	
	\begin{proof}
		Since $v=t_{u(i),u(j)}u$, we have $v^{-1}=u^{-1}t_{u(i),u(j)}$, that is, $u^{-1}$ and $v^{-1}$ are given as follows:
		\[
		\begin{cases}
		u^{-1}(k)=v^{-1}(k) &\text{if $k\not=u(i),u(j)$},\\
		u^{-1}(k)=i, \ v^{-1}(k)=j &\text{if $k=u(i)$},\\
		u^{-1}(k)=j,\ v^{-1}(k)=i &\text{if $k=u(j)$}.
		\end{cases}
		\]
		Hence the direction vector of the edge emanating from $\baru=(u^{-1}(1), \dots,u^{-1}(n))$ to $\barv=(v^{-1}(1),\dots,v^{-1}(n))$ is 
		\[
		\barv-\baru=(j-i)(e_{u(i)}-e_{u(j)}).
		\] 
		Since the inner product of $\barv-\baru$ with $(a_1,\dots,a_n)$ is $(j-i)(a_{u(i)}-a_{u(j)})$, it is positive if and only if $u(i)<u(j)$ because $i<j$ and $a_1>\dots>a_n$, proving the lemma. 
	\end{proof}
	
	\begin{remark}
		In the proof of \cite[Theorem 8.3]{le-ma20}, the authors say, 
		\begin{quote}
			``$e_{u(j)}-e_{u(i)}$'s for $(u(i), u(j)) \in E_w(u)$ are primitive inward-pointing edge vectors from the vertex $\mu(u)$.'' 
		\end{quote}
		\noindent
		However, this seems incorrect ($e_{u(j)}-e_{u(i)}$ should be $e_{u(i)}-e_{u(j)}$) although the conclusion is correct because the authors take the vector $(a_1,\dots,a_n)$ with $a_1<\dots<a_n$ there. 
	\end{remark}

	Motivated by Lemma~\ref{lemm:2-1}, we define 
	\begin{equation}
	\begin{split}
	E_w(u)^+&=\{ (u(i),u(j))\in E_w(u)\mid u(i)<u(j)\},\\
	E_w(u)^-&=\{ (u(i),u(j))\in E_w(u)\mid u(i)>u(j)\}.
	\end{split}
	\end{equation}
	Then 
	\[
	E_w(u)=E_w(u)^+\sqcup E_w(u)^-.
	\]
	Note that $E_w(u)^+$ (resp. $E_w(u)^-$) corresponds to ascending (resp. descending) edges emanating from $\baru$ by Lemma~\ref{lemm:2-1} with respect to the function $h$. 
	Now we set 
	\begin{equation}\label{eq_def_of_awu}
	a_w(u) \coloneqq  |E_w(u)^+|
	\end{equation}
	and define
	\begin{equation}\label{eq_def_of_Aw}
	A_w(t) \coloneqq \sum_{u \leq w} t^{a_w(u)}.
	\end{equation}
	With these notations, we present the main result of this article whose proof will be given in Section~\ref{section_proof}.
	\begin{theorem}\label{cor_main}
		Let $\Yw$ be the generic torus orbit closure of the Schubert variety~$\Xw$ for $w  \in \mathfrak{S}_n$. 
		Then the toric variety $\Yw$ is paved by affine spaces, and the Poincar\'e polynomial of $\Yw$ is given by 
		\[
		\P(\Yw,t) =A_w(t^2).
		\] 
	\end{theorem}
	As  mentioned in the introduction, the formula is established when $Y_w$ is smooth in~\cite{le-ma20}. Theorem~\ref{cor_main} shows that the formula holds even when $Y_w$ is singular.
	
	We enclose this section by presenting the Poincar\'{e} polynomials of  singular toric varieties $Y_{4231}$ and $Y_{3412}$ using Theorem~\ref{cor_main}. The moment polytopes of $Y_{4231}$ and $Y_{3412}$ are  singular Bruhat interval polytopes $\QQ_{id, 4231^{-1}}$ and $\QQ_{id, 3412^{-1}}$, respectively, see Figure~\ref{figure_BIP}.
	In both polytopes, we take $h \colon \R^4 \to \R$ by the inner product with a vector $(12,2,-1,-2)$. For instance, the image of the vertex $\mu(4132) = (2,4,3,1)$ under $h$ is given by
	\[
	h(\mu(4132)) = \langle (2,4,3,1), (12,2,-1,-2) \rangle = 2\times 12 + 4 \times 2 - 3 - 2 = 27.
	\]
	We label each vertex $\baru$ with $h(\mu(u))$.  
	Therefore, the vertex $\mu(4132)=(2,4,3,1)$ is labeled by $27$ (see Figure~\ref{figure_Bruhat_polytope_4231}).
	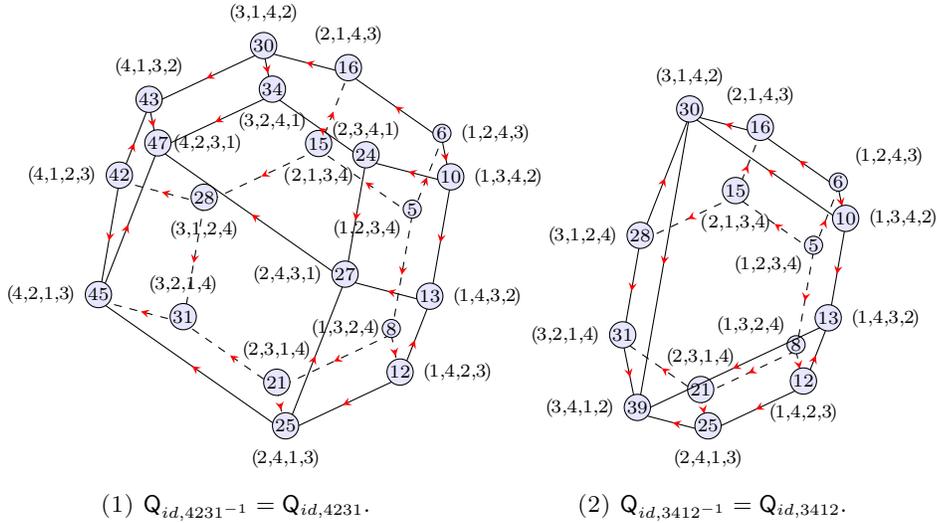
\begin{figure}[H]
		\begin{subfigure}[b]{0.5\textwidth}
			\centering
			\tikzset{  mid arrow/.style={postaction={decorate,decoration={
							markings,
							mark=at position .5 with {\arrow[#1]{stealth}}
			}}}}		\begin{tikzpicture}%
			[x={(-0.939161cm, 0.244762cm)},
			y={(0.097442cm, -0.482887cm)},
			z={(0.329367cm, 0.840780cm)},
			scale=1.20000,
			back/.style={dashed, thin},
			edge/.style={color=black},
			facet/.style={fill=none},
			vertex/.style={inner sep=0.3pt,circle, draw=black, fill=blue!10,anchor=base, font=\scriptsize},
			vertex2/.style={inner sep=1.5pt,circle,draw=green!25!black,fill=red,thick,anchor=base},
			every label/.append style={text=black, font=\scriptsize}]
			%
			%
			%% Coordinate of the vertices:
			%%
			\coordinate (1.00000, 2.00000, 3.00000) at (1.00000, 2.00000, 3.00000);
			\coordinate (1.00000, 2.00000, 4.00000) at (1.00000, 2.00000, 4.00000);
			\coordinate (1.00000, 3.00000, 2.00000) at (1.00000, 3.00000, 2.00000);
			\coordinate (1.00000, 3.00000, 4.00000) at (1.00000, 3.00000, 4.00000);
			\coordinate (1.00000, 4.00000, 2.00000) at (1.00000, 4.00000, 2.00000);
			\coordinate (1.00000, 4.00000, 3.00000) at (1.00000, 4.00000, 3.00000);
			\coordinate (2.00000, 1.00000, 3.00000) at (2.00000, 1.00000, 3.00000);
			\coordinate (2.00000, 1.00000, 4.00000) at (2.00000, 1.00000, 4.00000);
			\coordinate (2.00000, 3.00000, 1.00000) at (2.00000, 3.00000, 1.00000);
			\coordinate (2.00000, 3.00000, 4.00000) at (2.00000, 3.00000, 4.00000);
			\coordinate (2.00000, 4.00000, 1.00000) at (2.00000, 4.00000, 1.00000);
			\coordinate (2.00000, 4.00000, 3.00000) at (2.00000, 4.00000, 3.00000);
			\coordinate (3.00000, 1.00000, 2.00000) at (3.00000, 1.00000, 2.00000);
			\coordinate (3.00000, 1.00000, 4.00000) at (3.00000, 1.00000, 4.00000);
			\coordinate (3.00000, 2.00000, 1.00000) at (3.00000, 2.00000, 1.00000);
			\coordinate (3.00000, 2.00000, 4.00000) at (3.00000, 2.00000, 4.00000);
			\coordinate (4.00000, 1.00000, 2.00000) at (4.00000, 1.00000, 2.00000);
			\coordinate (4.00000, 1.00000, 3.00000) at (4.00000, 1.00000, 3.00000);
			\coordinate (4.00000, 2.00000, 1.00000) at (4.00000, 2.00000, 1.00000);
			\coordinate (4.00000, 2.00000, 3.00000) at (4.00000, 2.00000, 3.00000);
			%%
			%%
			%% Drawing edges in the back
			%%
			\draw[edge,back, mid arrow = red] (1.00000, 2.00000, 3.00000) -- (1.00000, 2.00000, 4.00000);
			\draw[edge,back, mid arrow = red] (1.00000, 2.00000, 3.00000) -- (1.00000, 3.00000, 2.00000);
			\draw[edge,back, mid arrow = red] (1.00000, 2.00000, 3.00000) -- (2.00000, 1.00000, 3.00000);
			\draw[edge,back, mid arrow = red] (1.00000, 3.00000, 2.00000) -- (1.00000, 4.00000, 2.00000);
			\draw[edge,back, mid arrow = red] (1.00000, 3.00000, 2.00000) -- (2.00000, 3.00000, 1.00000);
			\draw[edge,back, mid arrow = red] (2.00000, 1.00000, 3.00000) -- (2.00000, 1.00000, 4.00000);
			\draw[edge,back, mid arrow = red] (2.00000, 1.00000, 3.00000) -- (3.00000, 1.00000, 2.00000);
			\draw[edge,back, mid arrow = red] (2.00000, 3.00000, 1.00000) -- (2.00000, 4.00000, 1.00000);
			\draw[edge,back, mid arrow = red] (2.00000, 3.00000, 1.00000) -- (3.00000, 2.00000, 1.00000);
			\draw[edge,back, mid arrow = red] (3.00000, 1.00000, 2.00000) -- (3.00000, 2.00000, 1.00000);
			\draw[edge,back, mid arrow = red] (3.00000, 1.00000, 2.00000) -- (4.00000, 1.00000, 2.00000);
			\draw[edge,back, mid arrow = red] (3.00000, 2.00000, 1.00000) -- (4.00000, 2.00000, 1.00000);
			%%
			%%
			%% Drawing vertices in the back
			%%
			\node[vertex, label={[label distance = -3mu] left:{$(\mspace{-2mu}1,\!3,\!2,\!4\mspace{-2mu})$}}] at (1.00000, 3.00000, 2.00000)     {$8$};
			\node[vertex, label={above:{$(\mspace{-2mu}2,\!3,\!1,\!4\mspace{-2mu})$}}] at (2.00000, 3.00000, 1.00000)     {$21$};
			\node[vertex,label={[label distance = -3mu] below left:{$(\mspace{-2mu}1,\!2,\!3,\!4\mspace{-2mu})$}}] at (1.00000, 2.00000, 3.00000)     {$5$};
			\node[vertex,label={below:{$(2,\!1,\!3,\!4\mspace{-2mu})$}}] at (2.00000, 1.00000, 3.00000)     {$15$};
			\node[vertex,label={above:{$(\mspace{-2mu}3,\!2,\!1,\!4\mspace{-2mu})$}}] at (3.00000, 2.00000, 1.00000)     {$31$};
			\node[vertex,label={below:{$(\mspace{-2mu}3,\!1,\!2,\!4\mspace{-2mu})$}}] at (3.00000, 1.00000, 2.00000)     {$28$};
			%%
			%%
			%% Drawing the facets
			%%
			\fill[facet] (4.00000, 2.00000, 3.00000) -- (2.00000, 4.00000, 3.00000) -- (2.00000, 4.00000, 1.00000) -- (4.00000, 2.00000, 1.00000) -- cycle {};
			\fill[facet] (4.00000, 2.00000, 3.00000) -- (2.00000, 4.00000, 3.00000) -- (2.00000, 3.00000, 4.00000) -- (3.00000, 2.00000, 4.00000) -- cycle {};
			\fill[facet] (4.00000, 2.00000, 3.00000) -- (3.00000, 2.00000, 4.00000) -- (3.00000, 1.00000, 4.00000) -- (4.00000, 1.00000, 3.00000) -- cycle {};
			\fill[facet] (4.00000, 2.00000, 3.00000) -- (4.00000, 1.00000, 3.00000) -- (4.00000, 1.00000, 2.00000) -- (4.00000, 2.00000, 1.00000) -- cycle {};
			\fill[facet] (2.00000, 4.00000, 3.00000) -- (1.00000, 4.00000, 3.00000) -- (1.00000, 3.00000, 4.00000) -- (2.00000, 3.00000, 4.00000) -- cycle {};
			\fill[facet] (2.00000, 4.00000, 3.00000) -- (1.00000, 4.00000, 3.00000) -- (1.00000, 4.00000, 2.00000) -- (2.00000, 4.00000, 1.00000) -- cycle {};
			\fill[facet] (3.00000, 2.00000, 4.00000) -- (2.00000, 3.00000, 4.00000) -- (1.00000, 3.00000, 4.00000) -- (1.00000, 2.00000, 4.00000) -- (2.00000, 1.00000, 4.00000) -- (3.00000, 1.00000, 4.00000) -- cycle {};
			%%
			%%
			%% Drawing edges in the front
			%%
			\draw[edge, mid arrow = red] (1.00000, 2.00000, 4.00000) -- (1.00000, 3.00000, 4.00000);
			\draw[edge, mid arrow = red] (1.00000, 2.00000, 4.00000) -- (2.00000, 1.00000, 4.00000);
			\draw[edge, mid arrow = red] (1.00000, 3.00000, 4.00000) -- (1.00000, 4.00000, 3.00000);
			\draw[edge, mid arrow = red] (1.00000, 3.00000, 4.00000) -- (2.00000, 3.00000, 4.00000);
			\draw[edge, mid arrow = red] (1.00000, 4.00000, 2.00000) -- (1.00000, 4.00000, 3.00000);
			\draw[edge, mid arrow = red] (1.00000, 4.00000, 2.00000) -- (2.00000, 4.00000, 1.00000);
			\draw[edge, mid arrow = red] (1.00000, 4.00000, 3.00000) -- (2.00000, 4.00000, 3.00000);
			\draw[edge, mid arrow = red] (2.00000, 1.00000, 4.00000) -- (3.00000, 1.00000, 4.00000);
			\draw[edge, mid arrow = red] (2.00000, 3.00000, 4.00000) -- (2.00000, 4.00000, 3.00000);
			\draw[edge, mid arrow = red] (2.00000, 3.00000, 4.00000) -- (3.00000, 2.00000, 4.00000);
			\draw[edge, mid arrow = red] (2.00000, 4.00000, 1.00000) -- (2.00000, 4.00000, 3.00000);
			\draw[edge, mid arrow = red] (2.00000, 4.00000, 1.00000) -- (4.00000, 2.00000, 1.00000);
			\draw[edge, mid arrow = red] (2.00000, 4.00000, 3.00000) -- (4.00000, 2.00000, 3.00000);
			\draw[edge, mid arrow = red] (3.00000, 1.00000, 4.00000) -- (3.00000, 2.00000, 4.00000);
			\draw[edge, mid arrow = red] (3.00000, 1.00000, 4.00000) -- (4.00000, 1.00000, 3.00000);
			\draw[edge, mid arrow = red] (3.00000, 2.00000, 4.00000) -- (4.00000, 2.00000, 3.00000);
			\draw[edge, mid arrow = red] (4.00000, 1.00000, 2.00000) -- (4.00000, 1.00000, 3.00000);
			\draw[edge, mid arrow = red] (4.00000, 1.00000, 2.00000) -- (4.00000, 2.00000, 1.00000);
			\draw[edge, mid arrow = red] (4.00000, 1.00000, 3.00000) -- (4.00000, 2.00000, 3.00000);
			\draw[edge, mid arrow = red] (4.00000, 2.00000, 1.00000) -- (4.00000, 2.00000, 3.00000);
			%%
			%%
			%% Drawing the vertices in the front
			%%
			\node[vertex,label={right:{$(\mspace{-2mu}1,\!2,\!4,\!3\mspace{-2mu})$}}] at (1.00000, 2.00000, 4.00000)     {$6$};
			\node[vertex,label={right:{$(\mspace{-2mu}1,\!3,\!4,\!2\mspace{-2mu})$}}] at (1.00000, 3.00000, 4.00000)     {$10$};
			\node[vertex,label={right:{$(\mspace{-2mu}1,\!4,\!2,\!3\mspace{-2mu})$}}] at (1.00000, 4.00000, 2.00000)     {$12$};
			\node[vertex,label={right:{$(\mspace{-2mu}1,\!4,\!3,\!2\mspace{-2mu})$}}] at (1.00000, 4.00000, 3.00000)     {$13$};
			\node[vertex,label={above:{$(\mspace{-2mu}2,\!1,\!4,\!3\mspace{-2mu})$}}] at (2.00000, 1.00000, 4.00000)     {$16$};
			\node[vertex,label={ [label distance = -3mu]  above :{$(\mspace{-2mu}2,\!3,\!4,\!1\mspace{-2mu})$}}] at (2.00000, 3.00000, 4.00000)     {$24$};
			\node[vertex,label={below:{$(\mspace{-2mu}2,\!4,\!1,\!3\mspace{-2mu})$}}] at (2.00000, 4.00000, 1.00000)     {$25$};
			\node[vertex,label={ left:{$(\mspace{-2mu}2,\!4,\!3,\!1\mspace{-2mu})$}}] at (2.00000, 4.00000, 3.00000)     {$27$};
			\node[vertex,label={above:{$(\mspace{-2mu}3,\!1,\!4,\!2\mspace{-2mu})$}}] at (3.00000, 1.00000, 4.00000)     {$30$};
			\node[vertex,label={below:{$(\mspace{-2mu}3,\!2,\!4,\!1\mspace{-2mu})$}}] at (3.00000, 2.00000, 4.00000)     {$34$};
			\node[vertex,label={left:{$(\mspace{-2mu}4,\!1,\!2,\!3\mspace{-2mu})$}}]at (4.00000, 1.00000, 2.00000)     {$42$};
			\node[vertex,label={above:{$(\mspace{-2mu}4,\!1,\!3,\!2\mspace{-2mu})$}}] at (4.00000, 1.00000, 3.00000)     {$43$};
			\node[vertex,label={left:{$(\mspace{-2mu}4,\!2,\!1,\!3\mspace{-2mu})$}}] at (4.00000, 2.00000, 1.00000)     {$45$};
			\node[vertex,label={ [label distance = -3mu] 0:{$(\mspace{-2mu}4,\!2,\!3,\!1\mspace{-2mu})$}}] at (4.00000, 2.00000, 3.00000)     {$47$};
			
			\end{tikzpicture}%
			\caption{$\QQ_{id, {4231}^{-1}}=\QQ_{id, 4231}$.}
			\label{figure_Bruhat_polytope_4231}
		\end{subfigure}%
		\hspace{1.5em}
		\begin{subfigure}[b]{0.4\textwidth}
			\centering
			\tikzset{  mid arrow/.style={postaction={decorate,decoration={
							markings,
							mark=at position .5 with {\arrow[#1]{stealth}}
			}}}}
			\begin{tikzpicture}%
			[x={(-0.939161cm, 0.244762cm)},
			y={(0.097442cm, -0.482887cm)},
			z={(0.329367cm, 0.840780cm)},
			scale=1.00000,
			back/.style={dashed, thin},
			edge/.style={color=black},
			facet/.style={fill=none},
			vertex/.style={inner sep=0.3pt,circle, draw=black, fill=blue!10,anchor=base, font=\scriptsize},
			vertex2/.style={inner sep=1.5pt,circle,draw=green!25!black,fill=red,thick,anchor=base},
			every label/.append style={text=black, font=\scriptsize}]
			%
			%
			%% Coordinate of the vertices:
			%%
			\coordinate (1.00000, 2.00000, 3.00000) at (1.00000, 2.00000, 3.00000);
			\coordinate (1.00000, 2.00000, 4.00000) at (1.00000, 2.00000, 4.00000);
			\coordinate (1.00000, 3.00000, 2.00000) at (1.00000, 3.00000, 2.00000);
			\coordinate (1.00000, 3.00000, 4.00000) at (1.00000, 3.00000, 4.00000);
			\coordinate (1.00000, 4.00000, 2.00000) at (1.00000, 4.00000, 2.00000);
			\coordinate (1.00000, 4.00000, 3.00000) at (1.00000, 4.00000, 3.00000);
			\coordinate (2.00000, 1.00000, 3.00000) at (2.00000, 1.00000, 3.00000);
			\coordinate (2.00000, 1.00000, 4.00000) at (2.00000, 1.00000, 4.00000);
			\coordinate (2.00000, 3.00000, 1.00000) at (2.00000, 3.00000, 1.00000);
			\coordinate (2.00000, 4.00000, 1.00000) at (2.00000, 4.00000, 1.00000);
			\coordinate (3.00000, 1.00000, 2.00000) at (3.00000, 1.00000, 2.00000);
			\coordinate (3.00000, 1.00000, 4.00000) at (3.00000, 1.00000, 4.00000);
			\coordinate (3.00000, 2.00000, 1.00000) at (3.00000, 2.00000, 1.00000);
			\coordinate (3.00000, 4.00000, 1.00000) at (3.00000, 4.00000, 1.00000);
			%%
			%%
			%% Drawing edges in the back
			%%
			\draw[edge,back, mid arrow = red] (1.00000, 2.00000, 3.00000) -- (1.00000, 2.00000, 4.00000);
			\draw[edge,back, mid arrow = red] (1.00000, 2.00000, 3.00000) -- (1.00000, 3.00000, 2.00000);
			\draw[edge,back, mid arrow = red] (1.00000, 2.00000, 3.00000) -- (2.00000, 1.00000, 3.00000);
			\draw[edge,back, mid arrow = red] (1.00000, 3.00000, 2.00000) -- (1.00000, 4.00000, 2.00000);
			\draw[edge,back, mid arrow = red] (1.00000, 3.00000, 2.00000) -- (2.00000, 3.00000, 1.00000);
			\draw[edge,back, mid arrow = red] (2.00000, 1.00000, 3.00000) -- (2.00000, 1.00000, 4.00000);
			\draw[edge,back, mid arrow = red] (2.00000, 1.00000, 3.00000) -- (3.00000, 1.00000, 2.00000);
			\draw[edge,back, mid arrow = red] (2.00000, 3.00000, 1.00000) -- (2.00000, 4.00000, 1.00000);
			\draw[edge,back, mid arrow = red] (2.00000, 3.00000, 1.00000) -- (3.00000, 2.00000, 1.00000);
			%%
			%%
			%% Drawing vertices in the back
			%%
			\node[vertex, label={[label distance = -3mu] above left:{$(\mspace{-2mu}1,\!3,\!2,\!4\mspace{-2mu})$}}] at (1.00000, 3.00000, 2.00000)     {$8$};
			\node[vertex, label={above:{$(\mspace{-2mu}2,\!3,\!1,\!4\mspace{-2mu})$}}] at (2.00000, 3.00000, 1.00000)     {$21$};
			\node[vertex,label={[label distance = -3mu] below left:{$(\mspace{-2mu}1,\!2,\!3,\!4\mspace{-2mu})$}}] at (1.00000, 2.00000, 3.00000)     {$5$};
			\node[vertex,label={below:{$(2,\!1,\!3,\!4\mspace{-2mu})$}}] at (2.00000, 1.00000, 3.00000)     {$15$};
			
			%%
			%%
			%% Drawing the facets
			%%
			\fill[facet] (3.00000, 4.00000, 1.00000) -- (1.00000, 4.00000, 3.00000) -- (1.00000, 3.00000, 4.00000) -- (3.00000, 1.00000, 4.00000) -- cycle {};
			\fill[facet] (3.00000, 4.00000, 1.00000) -- (1.00000, 4.00000, 3.00000) -- (1.00000, 4.00000, 2.00000) -- (2.00000, 4.00000, 1.00000) -- cycle {};
			\fill[facet] (3.00000, 4.00000, 1.00000) -- (3.00000, 1.00000, 4.00000) -- (3.00000, 1.00000, 2.00000) -- (3.00000, 2.00000, 1.00000) -- cycle {};
			\fill[facet] (3.00000, 1.00000, 4.00000) -- (1.00000, 3.00000, 4.00000) -- (1.00000, 2.00000, 4.00000) -- (2.00000, 1.00000, 4.00000) -- cycle {};
			%%
			%%
			%% Drawing edges in the front
			%%
			\draw[edge, mid arrow = red] (1.00000, 2.00000, 4.00000) -- (1.00000, 3.00000, 4.00000);
			\draw[edge, mid arrow = red] (1.00000, 2.00000, 4.00000) -- (2.00000, 1.00000, 4.00000);
			\draw[edge, mid arrow = red] (1.00000, 3.00000, 4.00000) -- (1.00000, 4.00000, 3.00000);
			\draw[edge, mid arrow = red] (1.00000, 3.00000, 4.00000) -- (3.00000, 1.00000, 4.00000);
			\draw[edge, mid arrow = red] (1.00000, 4.00000, 2.00000) -- (1.00000, 4.00000, 3.00000);
			\draw[edge, mid arrow = red] (1.00000, 4.00000, 2.00000) -- (2.00000, 4.00000, 1.00000);
			\draw[edge, mid arrow = red] (1.00000, 4.00000, 3.00000) -- (3.00000, 4.00000, 1.00000);
			\draw[edge, mid arrow = red] (2.00000, 1.00000, 4.00000) -- (3.00000, 1.00000, 4.00000);
			\draw[edge, mid arrow = red] (2.00000, 4.00000, 1.00000) -- (3.00000, 4.00000, 1.00000);
			\draw[edge, mid arrow = red] (3.00000, 1.00000, 2.00000) -- (3.00000, 1.00000, 4.00000);
			\draw[edge, mid arrow = red] (3.00000, 1.00000, 2.00000) -- (3.00000, 2.00000, 1.00000);
			\draw[edge, mid arrow = red] (3.00000, 1.00000, 4.00000) -- (3.00000, 4.00000, 1.00000);
			\draw[edge, mid arrow = red] (3.00000, 2.00000, 1.00000) -- (3.00000, 4.00000, 1.00000);
			%%
			%%
			%% Drawing the vertices in the front
			%%
			\node[vertex,label={left:{$(\mspace{-2mu}3,\!2,\!1,\!4\mspace{-2mu})$}}] at (3.00000, 2.00000, 1.00000)     {$31$};
			\node[vertex,label={left:{$(\mspace{-2mu}3,\!1,\!2,\!4\mspace{-2mu})$}}] at (3.00000, 1.00000, 2.00000)     {$28$};

			\node[vertex,label={above right:{$(\mspace{-2mu}1,\!2,\!4,\!3\mspace{-2mu})$}}] at (1.00000, 2.00000, 4.00000)     {$6$};
			\node[vertex,label={ right:{$(\mspace{-2mu}1,\!3,\!4,\!2\mspace{-2mu})$}}] at (1.00000, 3.00000, 4.00000)     {$10$};
			\node[vertex,label={below:{$(\mspace{-2mu}1,\!4,\!2,\!3\mspace{-2mu})$}}] at (1.00000, 4.00000, 2.00000)     {$12$};
			\node[vertex,label={right:{$(\mspace{-2mu}1,\!4,\!3,\!2\mspace{-2mu})$}}] at (1.00000, 4.00000, 3.00000)     {$13$};
			\node[vertex,label={above:{$(\mspace{-2mu}2,\!1,\!4,\!3\mspace{-2mu})$}}] at (2.00000, 1.00000, 4.00000)     {$16$};
			\node[vertex,label={below:{$(\mspace{-2mu}2,\!4,\!1,\!3\mspace{-2mu})$}}] at (2.00000, 4.00000, 1.00000)     {$25$};
			%		\node[vertex,label={above left:{2314}}] at (3.00000, 1.00000, 2.00000)     {$3$};
			\node[vertex,label={above:{$(\mspace{-2mu}3,\!1,\!4,\!2\mspace{-2mu})$}}] at (3.00000, 1.00000, 4.00000)     {$30$};
			%		\node[vertex,label={above left:{3214}}] at (3.00000, 2.00000, 1.00000)     {$4$};
			\node[vertex,label={left:{$(\mspace{-2mu}3,\!4,\!1,\!2\mspace{-2mu})$}}] at (3.00000, 4.00000, 1.00000)     {$39$};
			
			\end{tikzpicture}
			\caption{$\QQ_{id,3412^{-1}}=\QQ_{id, 3412}$.}
			\label{figure_Bruhat_polytope_3412}
		\end{subfigure}
		\caption{Bruhat interval polytopes.}
		\label{figure_BIP}
	\end{figure}
	In Tables~\ref{table_retraction_seq_4231} and~\ref{table_retraction_seq_3412} (on pages~\pageref{table_retraction_seq_4231} and~\pageref{table_retraction_seq_3412}), we present retraction sequences for these Bruhat interval polytopes. Using this observation, we get the Poincar\'e polynomials of these singular generic torus orbit closures:
	\[
	\begin{split}
	\P(Y_{4231}) &= 1 + 7 t^2 + 11t^4 + t^6, \\
	\P(Y_{3412}) &= 1 + 5 t^2 + 7 t^4 + t^6.
	\end{split}
	\]
	Even though retraction sequences we presented in Tables~\ref{table_retraction_seq_4231} and~\ref{table_retraction_seq_3412} are different from those in~\cite[Tables~B.4 and B.5]{le-ma20}, we get the same Poincar\'e polynomials.
	%%%%
	\section{Proof of Theorem~\ref{cor_main}}
	\label{section_proof}
	In this section, we present the proof of Theorem~\ref{cor_main}.
	To give a proof, we  use retraction sequences (see Theorem~\ref{thm_poincare_polynomial_ascending_chains} and Proposition~\ref{prop_affine_paving}). 
	We shall show that for the linear function~$h$ in Lemma~\ref{lemm:2-1}, the following two conditions hold at each vertex $\baru$ of $\QQ_{id, w^{-1}}$:
	\begin{enumerate}
		\item[(C1)] the direction vectors of ascending edges emanating from $\baru$ are linearly independent, and
		\item[(C2)] the ascending edges emanating from $\baru$ produce a desired face of $\QQ_{id, w^{-1}}$. 
	\end{enumerate}
	We will verify these two conditions in Proposition~\ref{lemm:2-3} and Proposition~\ref{prop_form_a_face}, respectively. Combining these two propositions with Lemma~\ref{lemm:smooth-retraction}, the proof of Theorem~\ref{cor_main} follows.
	
	The following lemma is an easy consequence from the definition of $E_w(u)^{+}$, but quite useful.
	
	\begin{lemma}\label{lemm:ewu}
		If $(u(i),u(k))\in E_w(u)^{+}$, then there is no $j$ such that $i<j<k$ and $u(i)<u(j)<u(k)$. Similarly, if $(u(i),u(k))\in E_w(u)^{-}$, then there is no $j$ such that $i<j<k$ and $u(i)>u(j)>u(k)$.
	\end{lemma}
	\begin{proof}
		It follows from the requirement $|\ell(t_{u(i),u(j)}u)-\ell(u)|=1$. 
	\end{proof}

	Now we prove Condition (C1).
	
	\begin{lemma} \label{lemm:2-2}
		If $(u(i),u(j)), (u(i),u(k))\in E_w(u)^+$, then $j=k$.  Similarly, if $(u(i),u(k)), (u(j),u(k))\in E_w(u)^+$, then $i=j$.
	\end{lemma}
	
	\begin{proof}
		We will prove only the former statement because the latter statement can be proved similarly.
		Suppose that $j\neq k$. Then we may assume $j<k$ without loss of generality, and we have $i<j<k$. It follows from Lemma~\ref{lemm:ewu} that $u(j)>u(k)$. Then  there is a sequence 
		\[
		j=j_1<j_2<\cdots <j_p=k\quad\text{with}\quad 
		u(j)=u(j_1)>u(j_2)>\cdots >u(j_p)=u(k)
		\]
		such that for each $r=1,2,\dots, p-1$, there is no $j'$ with 
		\[
		j_r<j'<j_{r+1}\quad\text{and}\quad u(j_r)>u(j')>u(j_{r+1}).
		\]
		Then $(u(j_r),u(j_{r+1}))\in \widetilde{E}_w(u)$ for $r=1,2,\dots,p-1$. This together with the identity 
		\begin{equation} \label{eq:2-1}
		e_{u(i)}-e_{u(k)}= e_{u(i)} - e_{u(j)}+\sum_{r=1}^{p-1}(e_{u(j_r)}-e_{u(j_{r+1})})
		\end{equation}
		shows that $(u(i),u(k))$ is decomposable. Consequently $(u(i),u(k))\notin E_w(u)$, but this contradicts the assumption $(u(i),u(k))\in E_w(u)^+$. 
	\end{proof}
	
	The above lemma allows us to give an order on the elements of the set $E_w(u)^+$ as follows:
	\[
	E_w(u)^+=\{ (u(i_r),u(j_r))\mid r=1,2,\dots,a_w(u)\}
	\]
	such that 
	\begin{equation} \label{eq:2-4}
	u(i_1)<u(i_2)<\cdots <u(i_{a_w(u)}),
	\end{equation} where $a_w(u)=|E_w(u)^+|$, see~\eqref{eq_def_of_awu}.

	\begin{proposition} \label{lemm:2-3}
		The direction vectors of the edges corresponding to $E_w(u)^{+}$ are linearly independent.
	\end{proposition} 
	
	\begin{proof}
		Recall that the primitive direction vector of the edge corresponding to $(u(i_r),u(j_r))\in E_w(u)^+$ is $e_{u(i_r)}-e_{u(j_r)}$. 
		Suppose that $\sum_{r=1}^{a_w(u)} c_r(e_{u(i_r)}-e_{u(j_r)})=0$ with $c_r\in \R$. Then the coefficient of $e_{u(i_1)}$ at the left hand side is $c_1$ by \eqref{eq:2-4} since $u(i_r)<u(j_r)$ for each $r$. Therefore $c_1$ must be zero. Then the coefficient of $e_{u(i_2)}$ is $c_2$ and this must be zero. Repeating this argument, we see that all $c_r$'s are zero. Hence the set of primitive direction vectors $\{e_{u(i_r)}-e_{u(j_r)} \mid r=1,2,\dots,a_w(u)\}$ is linearly independent.
	\end{proof}

	For the remaining part of this section, we will concentrate on showing Condition~(C2).
	It is enough to find a function $F\colon \QQ_{id,w^{-1}}\subset \R^n\to \R$ satisfying~\eqref{eq_con_of_F}, see Remark~\ref{rmk_existence_of_F}. Such a function $F$ can be obtained by the function $f\colon [n]\to \R$ in the next lemma.
	
	\begin{lemma}\label{lemma_existence_of_f}
		For each vertex $\baru$ of $\QQ_{id, w^{-1}}$, there exists a function $f\colon [n]\to \R$ such that 
		\begin{equation} \label{eq:f}
		\begin{cases}
		f(p)=f(q) \quad&\text{if $(p,q)\in E_w(u)^+$},\\
		f(p)>f(q) \quad&\text{if $(p,q)\in E_w(u)^-$}.\end{cases}
		\end{equation} 
	\end{lemma}	
	We give the proof of the lemma above at the end of this section.
	\begin{proposition}\label{prop_form_a_face}
		For each vertex $\baru$ of the Bruhat interval polytope $\QQ_{id, w^{-1}}$, there is a face of  $\QQ_{id, w^{-1}}$ whose edges meeting at $\baru$ are the ascending edges emanating from~$\baru$. 	
	\end{proposition}
	
	\begin{proof}
		We define a linear function $F\colon \R^n\to \R$ by the inner product with the vector $(f(1),\dots,f(n))$.  Then 
		\[
		\begin{cases}
		F(e_p-e_q)=0 \quad&\text{if $(p,q)\in E_w(u)^+$},\\
		F(e_p-e_q)>0 \quad&\text{if $(p,q)\in E_w(u)^-$}.
		\end{cases}
		\]
		Therefore, the intersection of $\QQ_{id,w^{-1}}$ with the hyperplane $\baru+\ker F$ is the face of~$\QQ_{id,w^{-1}}$ containing the vertex $\baru$ whose edges meeting at $\baru$ are exactly the ascending edges emanating from $\baru$.  
	\end{proof}
	
	To prove Lemma~\ref{lemma_existence_of_f}, we need a series of lemmas. 
	We begin with the following observation. 
	
	\begin{lemma} \label{lemm:3-1}
		Let $(u(i),u(j)), (u(i'),u(j'))\in E_w(u)^+$. 
		Suppose that $i<i'<j$ and $u(i)>u(i')$. Then we have
		\begin{equation*}
		\begin{cases}
		u(i)>u(j') &\text{ if }j'<j,\text{ and}\\
		u(j)>u(j') &\text{ if }j<j'.
		\end{cases}
		\end{equation*}
	\end{lemma}
	We visualize the above lemma in Figure~\ref{fig:ascending}.
	\begin{figure}[ht]
		\begin{subfigure}[c]{0.33\textwidth}
			\centering
			\begin{tikzpicture}[scale=.5]
			\draw[->] (0,0)--(5,0);
			\draw[->] (0,0)--(0,5);
			\draw (0,5) node[above] {$u(s)$};
			\draw (5,0) node[below] {$s$};
			\filldraw (1,3) circle(2pt) node[below]  {\footnotesize$u(i)$};
			\filldraw (2,1.5) circle(2pt) node [below] {\footnotesize$u(i')$};
			\filldraw (3,2.5) circle(2pt) node  [below] {\footnotesize$u(j')$};
			\filldraw (4,4.5) circle(2pt) node  [below] {\footnotesize$u(j)$};
			\end{tikzpicture}
			\subcaption{$i<i'<j'<j$.}
		\end{subfigure}
		\begin{subfigure}[c]{0.6\textwidth}
			\centering
			\begin{tikzpicture}[scale=.5]
			\draw[->] (0,0)--(5,0);
			\draw[->] (0,0)--(0,5);
			\draw (0,5) node[above] {$u(s)$};
			\draw (5,0) node[below] {$s$};
			\filldraw (1,3) circle(2pt) node[below]  {\footnotesize$u(i)$};
			\filldraw (2,1.5) circle(2pt) node [below] {\footnotesize$u(i')$};
			\filldraw (4,3.5) circle(2pt) node  [below] {\footnotesize$u(j')$};
			\filldraw (3,4.5) circle(2pt) node  [below] {\footnotesize$u(j)$};
			\draw (6,2.5) node{\footnotesize{or}};
			\draw[->] (0+7.5,0)--(5+7.5,0);
			\draw[->] (0+7.5,0)--(0+7.5,5);
			\draw (0+7.5,5) node[above] {$u(s)$};
			\draw (5+7.5,0) node[below] {$s$};
			\filldraw (1+7.5,3) circle(2pt) node[below]  {\footnotesize$u(i)$};
			\filldraw (2+7.5,1.5) circle(2pt) node [below] {\footnotesize$u(i')$};
			\filldraw (4+7.5,2) circle(2pt) node  [below] {\footnotesize$u(j')$};
			\filldraw (3+7.5,4.5) circle(2pt) node  [below] {\footnotesize$u(j)$};
			\end{tikzpicture}
			\subcaption{$i<i'<j<j'$.}
		\end{subfigure}
		\caption{Visualization of Lemma~\ref{lemm:3-1}.}
		\label{fig:ascending}
	\end{figure}
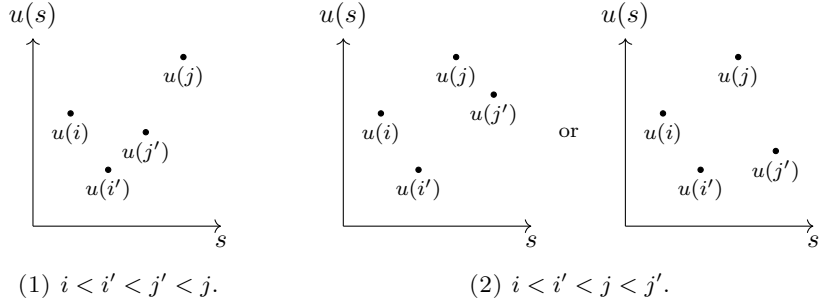
	\begin{proof}[Proof of Lemma~\ref{lemm:3-1}]
		From the hypotheses, we have $u(i')<u(i)<u(j)$.
		First assume $j'<j$. Then $i<i'<j'<j$. Since 
		$(u(i),u(j))\in E_w(u)^+$, we have $u(i)>u(j')$ or $u(j')>u(j)$  by Lemma~\ref{lemm:ewu}. However, the latter does not occur. Indeed, if $u(j')>u(j)$, then we get  descending paths from $u(j')$ to $u(j)$ and from $u(i)$ to $u(i')$ since $i<i'$ and $u(i)>u(i')$. Hence $(u(i),u(j))$ becomes a decomposable element which is a contradiction. Thus, $u(i)>u(j')$ if $j'<j$.
		
		Now assume $j<j'$. Then $i'<j<j'$. Since $(u(i'),u(j'))\in E_w(u)^+$, it follows from Lemma~\ref{lemm:ewu} that $u(j)>u(j')$. 
	\end{proof}
	
	To each set $E_w(u)^+$, we   associate a graph $\Gamma_w(u)^+$ as follows:
	\begin{enumerate}
		\item the vertices of $\Gamma_w(u)^+$ are labeled with $i$ and placed at the position $(i,u(i))$ for $i=1,\dots,n$, and
		\item the vertices $i$ and $j$ are joined by an edge if and only if $(u(i),u(j))\in E_w(u)^+$.
	\end{enumerate}
	Lemma~\ref{lemm:2-2} says that every connected component of the graph is a path (or a point). 
	For an ordered set $I=(i_1,\dots,i_r)$ with $1\le i_1<\dots<i_r\le n$, we call $I$ an \emph{ascent} in $E_w(u)^+$ if $(u(i_k),u(i_{k+1}))\in E_w(u)^+$ for every $k=1,\dots,r-1$. If an ascent $I$ corresponds to a connected component of the graph $\Gamma_w(u)^+$, then we call it \emph{maximal}. 
	For two maximal ascents $I$ and $I'$ in $E_w(u)^+$, we define a set $(I,I')$ of ordered pairs by 
	\[
	(I,I')\coloneqq\{(i,i')\mid i\in I,\ i'\in I',\ i<i',\ u(i)>u(i')\}.
	\]
	
	Note that both $(I,I')$ and $(I',I)$ may be empty. If both $(I,I')$ and $(I',I)$ are non-empty, then two paths corresponding to $I$ and $I'$ are crossing. In the following, we show that there is no crossing in the graph $\Gamma_w(u)^+$. For example, in Figure~\ref{fig:crossing}, two ascents $I=(i_1,i_2,i_3,i_4)$ and $I'=(i_1',i_2',i_3',i_4')$ are crossing, $(I,I')=\{(i_1,i_1')\}$ and $(I',I)=\{(i_2',i_3),(i_3',i_4)\}$. However, this does not occur because if $(i_1,i_2),(i_1',i_2')\in E_w(u)^+$ and $i_1<i_1'<i_2<i_2'$, then $u(i_2')<u(i_2)$  by Lemma~\ref{lemm:3-1}.
	
	\begin{figure}
		\centering
		\begin{tikzpicture}[scale=.5,node/.style={circle,draw, fill=white!20, inner sep = 0.25mm}]
		\draw[->] (0,0.5)--(10,0.5);
		\draw[->] (0,0.5)--(0,9);
		\draw (0,9) node[left] {$u(s)$};
		\draw (10,0.5) node[below] {$s$};
		\node[node] (1) at (1,2.5) {$i_1$};
		\node[node] (2) at (3,4) {$i_2$};
		\node[node] (3) at (6,5) {$i_3$};
		\node[node] (4) at (8,6) {$i_4$};
		
		\draw (1) to  (2);
		\draw (2) to  (3);
		\draw (3) to  (4);
		
		\node[node] (6) at (2.5,2) {$i_1'$};
		\node[node] (7) at (5,5.5) {$i_2'$};
		\node[node] (8) at (6.5,7) {$i_3'$};
		\node[node] (9) at (9,8) {$i_4'$};
		
		\draw (6) to  (7);
		\draw (7) to  (8);
		\draw (8) to  (9);

		\end{tikzpicture}
		\caption{Crossing ascents.}\label{fig:crossing}
	\end{figure}
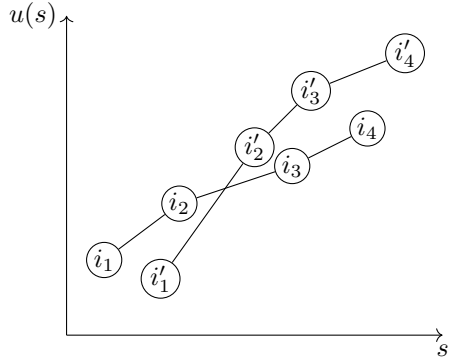

	\begin{lemma} \label{lemm:3-2}
		Either $(I,I')$ or $(I',I)$ is empty.
	\end{lemma}
	
	\begin{proof}
		Suppose that both sets are non-empty. We will deduce a contradiction. If $(i,i')\in (I,I')$ and $(j',j)\in (I',I)$, then $i<j$ or $j < i$. Furthermore, if $i<j$, then $i'<j'$ since $i<i'$, $j'<j$ and $u(i')<u(i)<u(j)<u(j')$. Similarly, if $j < i$, then $j' < i'$.
		Therefore, one of the following two sets are non-empty: 
		\[
		\begin{split}
		&\{ (i,i',j',j)\mid (i,i')\in (I,I'),\ (j',j)\in (I',I),\ i<j\} \quad\text{and}\\
		&\{ (j',j,i,i')\mid (j',j)\in (I',I),\ (i,i')\in (I,I'),\ j'<i'\}.
		\end{split}
		\]
		Since the roles of $I$ and $I'$ are the same, without loss of generality we assume that the former set above is non-empty. That is,
		\[
		(\!(I,I')\!)\coloneqq\{ (i,i',j',j)\mid (i,i')\in (I,I'),\ (j',j)\in (I',I),\ i<j\} \neq\emptyset.
		\]
		Note that $(i,i',j',j)\in (\!(I,I')\!)$ if and only if $i<i'<j'<j$ and $u(i')<u(i)<u(j)<u(j')$. 
		
		Since $(\!(I,I')\!)$ is a finite set, we can take a \emph{minimal} element $(i,i',j',j)\in (\!(I,I')\!)$ in the sense that if $(k,k',\ell',\ell)\in (\!(I,I')\!)$ satisfies $i\le k, i'\le k', \ell'\le j', \ell\le j$, then $(k,k',\ell',\ell)=(i,i',j',j)$. Since $i<i'$ and $u(i)>u(i')$, there is a descending path from the vertex $i$  to the vertex $i'$. Similarly, since $j'<j$ and $u(j')>u(j)$, there is a descending path from the vertex $j'$ to the vertex $j$. Moreover, there is an ascending path from the vertex $i'$ to the vertex $j'$ since $i',j'\in I'$ with $i'<j'$. These three paths connect the vertex $i$ and the vertex $j$, so $(u(i),u(j))\notin E_w(u)^+$. Therefore, there is $k\in I$ with $i<k<j$ such that $(u(i),u(k))\in E_w(u)^+$. Then $u(i)<u(k)<u(j)$. 
		
		If $i<k<i'$, then $k<i'<j'<j$ and $u(i')<u(i)<u(k)<u(j)<u(j')$. It means that $(k,i',j',j)\in (\!(I,I')\!)$, but this contradicts the minimality of $(i,i',j',j)$. Similarly, we can exclude the case $j'<k<j$.
		
		In the sequel we may assume $i'<k<j'$. Then $u(i')<u(k)<u(j')$ since $u(i')<u(i)<u(k)<u(j)<u(j')$). Hence $(u (i'), u (j')) \not\in E_w(u)^+$ by Lemma~\ref{lemm:ewu}. However, since $i',j' \in I'$, there must exist $k' \in I'$ such that $i' < k' < j'$ and $u(i') < u(k') < u(j')$.
		
		We may further assume $(u(i'),u(k'))\in E_w(u)^+$. If $k'<k$, then $i<i'<k'<k$. Since both $(u(i),u(k))$ and $(u(i'),u(k'))$ are in $E_w(u)^+$, we have $u(k')<u(i)$ by Lemma~\ref{lemm:3-1}. Therefore $(i,k',j',j)\in (\!(I,I')\!)$ but this contradicts the minimality of $(i,i',j',j)$. If $k<k'$, then $i<i'<k<k'$. Since both $(u(i),u(k))$ and $(u(i'),u(k'))$ are in $E_w(u)^+$, we have $u(k')<u(k)$ by Lemma~\ref{lemm:3-1}. Therefore $(k,k',j',j)\in (\!(I,I')\!)$ but this also contradicts the minimality of $(i,i',j',j)$. This completes the proof of the lemma. 
	\end{proof}

	\begin{proof}[Proof of Lemma~\ref{lemma_existence_of_f}]
		We denote by $u(I)$ the ascending path associated to an ascent $I$, that is, if $I=(i_1,\dots,i_r)$, then $u(I)=(u(i_1),\dots,u(i_r))$. Let $P_w(u)^+$ be the set of all maximal ascending paths in $\Gamma_w(u)^+$. For $u(I),u(I')\in P_w(u)^+$, we define $u(I)>u(I)'$ if there are $i\in I$ and $i'\in I'$ such that $i<i'$ and $u(i)>u(i')$. The order $u(I)>u(I')$ is well-defined by Lemma~\ref{lemm:3-2}. We extend this partial order on $P_w(u)^+$ to a total order and take a function $f\colon P_w(u)^+\to \R$ preserving the total order. Since $P_w(u)^+$ defines a partition of $[n]$, the function $f$ induces a function on $[n]$, which we also denote by $f$. This function $f$ satisfies the condition in \eqref{eq:f}. Indeed, the former condition that $f(p)=f(q)$ if $(p,q)\in E_w(u)^+$ is obvious from our construction of $f$. To see the latter condition, we write $p=u(i)$ and $q=u(j)$ for $(p,q)\in E_w(u)^-$. Then $i<j$ and $u(i)>u(j)$. Therefore, the maximal ascending path containing $p=u(i)$ is bigger than that containing $q=u(j)$ with respect to the total order on $P_w(u)^+$. Hence $f(p)>f(q)$ since $f$ is an order preserving function. This proves the existence of the function $f$ in \eqref{eq:f}. 
	\end{proof}
	
	%%%%%
	\section{Concluding remarks}
	\label{sec_concluding_remarks}
	
	The study of generic torus orbit closures in Schubert varieties is related to combinatorics of Bruhat interval polytopes and the recent study of retraction sequences of polytopes. In this paper, we have observed retraction sequences of a certain class of Bruhat interval polytopes and their combinatorics. Below, we briefly mention some possible avenues for further exploration.
	
	The generic torus orbit closures in Schubert varieties are related to Bruhat interval polytopes $\QQ_{id, w^{-1}}$. The polytopes $\QQ_{id, w^{-1}}$ are a special class of the Bruhat interval polytopes. In general, for $v,w \in \mathfrak{S}_n$ with $v\le w$, the Bruhat interval polytope $\QQ_{v^{-1},w^{-1}}$ is related to the generic torus orbit closure in the \defi{Richardson variety} $X^{v}_{w} \coloneqq X_{w} \cap w_0X_{w_0v}$. Here, $w_0$ is the longest element in $\mathfrak{S}_n$. 
	
	If the Bruhat interval polytope $\QQ_{v^{-1},w^{-1}}$ admits a retraction sequence with respect to a linear function $h \colon \R^n \to \R$ defined by the inner product with a vector $(a_1,\dots,a_n) \in \R^n$ with $a_1 > \cdots > a_n$, then the Poincar\'e polynomial of the generic torus orbit closure $Y_{v,w}$ in $X^{v}_{w}$ can be expressed by the polynomial
	\[
	A_{v,w}(t) \coloneqq \sum_{v \leq u \leq w} t^{a_{v,w}(u)}
	\]
	where $a_{v,w}(u)$ is the number of ascending edges emanating from the vertex $\baru$ of~$\QQ_{v^{-1},w^{-1}}$. Indeed, we have
	\[
	\P(Y_{v,w},t) = A_{v,w}(t^2).
	\]

	However, not every Bruhat interval polytope admits a retraction sequence in general. 
	For instance, the Bruhat interval polytope $\QQ_{1324, 4231}$ is not a simple polytope and there are eight simple vertices. See Figure~\ref{figure_BIP_1324_4231}. It is not difficult to check that there is no retraction sequence. Consequently, we propose the following question.
	\begin{Question}\label{question_1}
		Which pair of permutations $v$ and $w$ guarantees that the Bruhat interval polytope $\QQ_{v,w}$ admits a retraction sequence? 
	\end{Question}	
	
	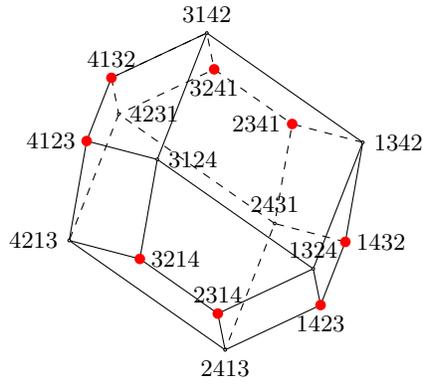
\begin{figure}[t]
		\begin{tikzpicture}%
		[x={(-0.939161cm, 0.244762cm)},
		y={(0.097442cm, -0.482887cm)},
		z={(0.329367cm, 0.840780cm)},
		scale=1,
		back/.style={dashed, thin},
		edge/.style={color=black},
		facet/.style={fill=blue!95!black, fill opacity=0.100000},
		vertex/.style={inner sep=0.3pt,circle, draw=black, fill=blue!10,anchor=base, font=\scriptsize},
		vertex2/.style={inner sep=1.5pt,circle,draw=green!25!black,fill=red,thick,anchor=base},
		every label/.append style={text=black, font=\small}]
		%
		%
		%% Coordinate of the vertices:
		%%
		\coordinate (123) at (1.00000, 2.00000, 3.00000);
		\coordinate (124) at (1.00000, 2.00000, 4.00000);
		\coordinate (132) at (1.00000, 3.00000, 2.00000);
		\coordinate (134) at (1.00000, 3.00000, 4.00000);
		\coordinate (142) at (1.00000, 4.00000, 2.00000);
		\coordinate (143) at (1.00000, 4.00000, 3.00000);
		\coordinate (213) at (2.00000, 1.00000, 3.00000);
		\coordinate (214) at (2.00000, 1.00000, 4.00000);
		\coordinate (231) at (2.00000, 3.00000, 1.00000);
		\coordinate (234) at (2.00000, 3.00000, 4.00000);
		\coordinate (241) at (2.00000, 4.00000, 1.00000);
		\coordinate (243) at (2.00000, 4.00000, 3.00000);
		\coordinate (312) at (3.00000, 1.00000, 2.00000);
		\coordinate (314) at (3.00000, 1.00000, 4.00000);
		\coordinate (321) at (3.00000, 2.00000, 1.00000);
		\coordinate (324) at (3.00000, 2.00000, 4.00000);
		\coordinate (342) at (3,4,2);
		\coordinate (341) at (3,4,1);
		\coordinate (412) at (4.00000, 1.00000, 2.00000);
		\coordinate (413) at (4.00000, 1.00000, 3.00000);
		\coordinate (431) at (4,3,1);
		\coordinate (421) at (4.00000, 2.00000, 1.00000);
		\coordinate (423) at (4.00000, 2.00000, 3.00000);
		\coordinate (432) at (4,3,2);

		\draw[edge] (312)--(314)--(134)--(132)--cycle;
		\draw[edge] (314)--(413)--(412)--(312);
		\draw[edge] (312)--(321)--(231)--(132);
		\draw[edge] (321)--(421)--(412);
		\draw[edge] (421)--(241)--(231);
		\draw[edge] (241)--(142)--(132);
		\draw[edge] (134)--(143)--(142);

		\draw[edge, back] (134)--(234)--(243);
		\draw[edge, back] (143)--(243)--(241);
		\draw[edge, back] (243)--(423)--(421);
		\draw[edge, back] (413)--(423)--(324);
		\draw[edge, back] (413)--(314)--(324)--(234);
		
		\node[vertex, label={left:{4123}}] at (412) {};
		\node[vertex, label= {above:{4132}}] at (413) {};
		\node[vertex, label= {right:{3124}}] at (312) {};
		\node[vertex, label = {above:{3142}}] at (314) {};
		\node[vertex, label={left:{4213}}] at (421) {};
		\node[vertex, label = {right:{4231}}] at (423) {};
		\node[vertex, label = {right:{3214}}] at (321) {};
		\node[vertex, label = {below:{3241}}] at (324) {};
		\node[vertex, label = {above:{2314}}] at (231) {};
		\node[vertex, label = {left:{2341}}] at (234) {};
		\node[vertex, label = {below:{2413}}] at (241) {};
		\node[vertex, label = {above:{2431}}] at (243) {};
		\node[vertex, label = {above:{1324}}] at (132) {};
		\node[vertex, label = {right:{1342}}] at (134) {};
		\node[vertex, label = {below:{1423}}] at (142) {};
		\node[vertex, label = {right:{1432}}] at (143) {};
		
		\fill[red] (321) circle (2pt);
		\fill[red] (231) circle (2pt);\fill[red] (234) circle (2pt);\fill[red] (324) circle (2pt);\fill[red] (412) circle (2pt);\fill[red] (413) circle (2pt);\fill[red] (142) circle (2pt);\fill[red] (143) circle (2pt);
		\end{tikzpicture}
		\caption{Bruhat interval polytope $\QQ_{1324, 4231}$ with vertices labeled by permutations $1324 \leq u \leq 4231$, where the red-colored vertices are simple.}
		\label{figure_BIP_1324_4231}
	\end{figure}

	The vertex $\overline{id}=(1,\dots,n)$ of $\QQ_{id, w^{-1}}$ is smooth for any $w \in \mathfrak{S}_n$ (see~\cite[Corollary~7.13]{le-ma20}), though the vertex $\bar{w}=(w^{-1}(1),\dots,w^{-1}(n))$ is not necessarily smooth. The smoothness of the vertex $\bar{w}$ can be determined in terms of graphs (see~\cite[Theorem~1.2]{le-ma20}). It is conjectured in~\cite[Conjecture~7.17]{le-ma20} that the Bruhat interval polytope $\QQ_{id, w^{-1}}$ is smooth if it is smooth at the vertex~$\bar{w}$. (This is equivalent to saying that the generic torus orbit closure $\Yw$ in $\Xw$ is smooth if it is smooth at the fixed point $wB$.) This conjecture  is related to the following question.

	\begin{Question}
		If the set $\left\{e_{w(i)}-e_{w(j)} \mid (w(i),w(j))\in E_w(w)^-=E_w(w)\right\}$ is linearly independent \textup{(}this is equivalent to $wB$ being a smooth point in $\Yw$\textup{)}, then is the set 
		\[
		\left\{e_{u(i)}-e_{u(j)} \mid (u(i),u(j))\in E_w(u)^-\right\}
		\]
		also linearly independent for each $u\le w$? 
	\end{Question}

	\begin{landscape}
		\begin{table}
			\begin{flushleft}
\begin{tabular}{lccccccc} \toprule 
					$u$ & $1234$ & $1243$ & $1324$ & $1423$ & $1342$  & $1432$  & $2134$ \\ \midrule 
					$\mu(u)$ & $(1,2,3,4)$ & $(1,2,4,3)$ & $(1,3,2,4)$ & $(1,3,4,2)$ & $(1,4,2,3)$  & $(1,4,3,2)$ & $(2,1,3,4)$  \\
					$h(\mu(u))$ & $5$ & $6$  & $8$  & $10$  & $12$ & $13$& $15$  \\
					$a_w(u)$ & $3$ & $2$ & $2$ & $2$ & $2$ & $1$ & $2$  \\
\bottomrule
				\end{tabular}
				
				\vspace{1em}
				
				\begin{tabular}{lccccccc}		\toprule
					$u$ & $2143$ & $3124$ & $4123$ & $3142$  & $4132$ & $2314$ & $2413$ \\ \midrule
					$\mu(u)$ & $(2,1,4,3)$ & $(2,3,1,4)$ & $(2,3,4,1)$ & $(2,4,1,3)$  & $(2,4,3,1)$ & $(3,1,2,4)$ & $(3,1,4,2)$ \\
					$h(\mu(u))$ & $16$ & $21$ & $24$ & $25$ & $27$  & $28$ & $30$ \\
					$a_w(u)$ & $1$ & $2$ & $2$ &$2$ & $1$& $2$ & $2$  \\
\bottomrule
				\end{tabular}
				
				\vspace{1em}
				
				\begin{tabular}{lcccccc} \toprule
					$u$ & $3214$ & $4213$ & $2341$ & $2431$ & $3241$   & $4231$ \\ \midrule
					$\mu(u)$ & $(3,2,1,4)$ & $(3,2,4,1)$ & $(4,1,2,3)$ & $(4,1,3,2)$ & $(4,2,1,3)$   & $(4,2,3,1)$ \\
					$h(\mu(u))$ & $31$ & $34$ & $42$ & $43$ & $45$  & $47$\\
					$a_w(u)$ & $1$ & $1$ & $2$ & $1$ & $1$ & $0$ \\
\bottomrule
				\end{tabular}
			\end{flushleft}
			\caption{A retraction sequence of the Bruhat interval polytope $\QQ_{id, 4231^{-1}}$.}
			\label{table_retraction_seq_4231}
		\end{table}
		
		\begin{table}
			\begin{center}
				\begin{tabular}{lccccccc} \toprule
					$u$ & $1234$   & $1243$ & $1324$   & $1423$ & $1342$  & $1432$ & $2134$\\ \midrule
					$\mu(u)$ & $(1,2,3,4)$  & $(1,2,4,3)$ & $(1,3,2,4)$   & $(1,3,4,2)$ & $(1,4,2,3)$&$(1,4,3,2)$ &$(2,1,3,4)$ \\
					$h(\mu(u))$ & $5$ & $6$ & $8$ & $10$ & $12$ & $13$ & $15$ \\
					$a_w(u)$ & $3$ & $2$ & $2$ & $2$ & $2$ & $1$ & $2$ \\
\bottomrule
				\end{tabular}
				
				\vspace{1em}
				
				\begin{tabular}{lccccccc} \toprule
					$u$ &$2143$ & $3124$ & $3142$ & $2314$ & $2413$ & $3214$ & $3412$\\ \midrule
					$\mu(u)$ & $(2,1,4,3)$ & $(2,3,1,4)$ & $(2,4,1,3)$ & $(3,1,2,4)$ & $(3,1,4,2)$ & $(3,2,1,4)$ & $(3,4,1,2)$ \\
					$h(\mu(u))$ & $16$ & $21$ & $25$ & $28$ & $30$ & $31$ & $39$\\
					$a_w(u)$ & $1$ & $2$ & $1$ & $2$ & $1$ & $1$ & $0$ \\
\bottomrule
				\end{tabular}
			\end{center}
			\caption{A retraction sequence of the Bruhat interval polytope $\QQ_{id, 3412^{-1}}$.}
			\label{table_retraction_seq_3412}
			
		\end{table}
	\end{landscape}

	\bibliographystyle{amsplain}
	\bibliography{ref}
	
\end{document}